\def\DateTime{6/September/2008, 9:00(Kyoto)}
\def\Version{Version $1.0$}
\def\yes{\if00}
\def\no{\if01}
\def\iftwelvept{\no}

\def\ifusepdf{\no}
\def\ifpsfont{\yes}

\iftwelvept
\documentclass[leqno,12pt]{amsart}
\else
\documentclass[leqno]{amsart}
\fi
\usepackage{amssymb}
\usepackage{amscd}
\usepackage{latexsym}
\usepackage{verbatim}
\usepackage[all]{xy}
\usepackage{color}
\usepackage{enumerate}

\ifusepdf
\usepackage{hyperref}
\else\fi
\ifpsfont
\usepackage[T1]{fontenc}
\usepackage{times}
\else\fi

\iftwelvept
\else\fi

\theoremstyle{plain}
\newtheorem{Theorem}{Theorem}[section]
\newtheorem{Proposition}[Theorem]{Proposition}
\newtheorem{Lemma}[Theorem]{Lemma}
\newtheorem{Corollary}[Theorem]{Corollary}

\newtheorem{Claim}{Claim}[Theorem]

\theoremstyle{definition}

\newtheorem{Example}[Theorem]{Example}

\renewcommand{\theTheorem}{\arabic{section}.\arabic{Theorem}}
\renewcommand{\theClaim}{\arabic{section}.\arabic{Theorem}.\arabic{Claim}}
\renewcommand{\theequation}{\arabic{section}.\arabic{Theorem}.\arabic{Claim}}

\def\rom{\textup}
\newcommand{\ZZ}{{\mathbb{Z}}}
\newcommand{\QQ}{{\mathbb{Q}}}
\newcommand{\RR}{{\mathbb{R}}}
\newcommand{\CC}{{\mathbb{C}}}

\newcommand{\KK}{{\mathbb{K}}}

\newcommand{\OO}{{\mathcal{O}}}

\newcommand{\Pic}{\operatorname{Pic}}

\newcommand{\voldiff}{\gamma}
\newcommand{\Image}{\operatorname{Image}}
\newcommand{\Ker}{\operatorname{Ker}}

\newcommand{\Spec}{\operatorname{Spec}}

\newcommand{\Supp}{\operatorname{Supp}}

\newcommand{\rank}{\operatorname{rk}}

\newcommand{\aPic}{\widehat{\operatorname{Pic}}}

\newcommand{\avol}{\widehat{\operatorname{vol}}}
\newcommand{\ah}{\hat{h}^0}

\newcommand{\achi}{\hat{\chi}}
\newcommand{\zeros}{\operatorname{div}}

\newcommand{\vol}{\operatorname{vol}}
\newcommand{\dist}{\operatorname{dist}}
\newcommand{\sub}{\operatorname{sub}}
\newcommand{\quot}{\operatorname{quot}}

\newcommand{\rest}[2]{\left.{#1}\right\vert_{{#2}}}

\begin{document}

\title[Continuous extension of arithmetic volumes]%
{Continuous extension of arithmetic volumes}
\author{Atsushi Moriwaki}
\address{Department of Mathematics, Faculty of Science,
Kyoto University, Kyoto, 606-8502, Japan}
\email{moriwaki@math.kyoto-u.ac.jp}
\date{\DateTime, (\Version)}
\dedicatory{To the memory of the late Professor Masayoshi Nagata}
\begin{abstract}
This paper is the sequel of the paper \cite{MoCont},
in which we established the arithmetic volume
function of $C^{\infty}$-hermitian $\QQ$-invertible sheaves and
proved its continuity.
The continuity of the volume function has a lot of applications as
treated in \cite{MoCont}.
In this paper, we would like to consider its continuous extension over $\RR$.
\end{abstract}


\maketitle

\setcounter{tocdepth}{1}
\tableofcontents

\section*{Introduction}
\renewcommand{\theTheorem}{\Alph{Theorem}}

Let $X$ be a $d$-dimensional projective arithmetic variety.
In \cite{MoCont}, 
for a $C^{\infty}$-hermitian invertible sheaf $\overline{L}$ on $X$,
we introduce the arithmetic volume $\avol(\overline{L})$
defined by
\[
\avol(\overline{L}) := \limsup_{n\to\infty} \frac{\log \# \{ s \in H^0(X, nL) \mid
\Vert s \Vert_{\sup} \leq 1 \} }{n^d/d!}.
\]
By Chen's recent work \cite{Chen}, ``$\limsup$'' in the above equation can be replaced by
``$\lim$''. Moreover, in \cite{MoCont}, we construct a positively homogeneous function
$\avol : \aPic(X)\otimes_{\ZZ} \QQ \to \RR$ of degree $d$ such that
the following diagram is commutative:
\[
\xymatrix{
\aPic(X)  \ar[r]^(.54){\avol} \ar[d] & \RR \\
\aPic(X) \otimes_{\ZZ} \QQ \ar[ru]_{\avol} & \\
}
\]
The most important result of \cite{MoCont} is the continuity of
$\avol : \aPic(X)\otimes_{\ZZ} \QQ \to \RR$, which has a lot of applications as treated in \cite{MoCont}.
In this paper, we would like to consider its continuous extension  over $\RR$, 
which is not obvious because
a continuous and positively homogeneous function on a vector space over $\QQ$
does not necessarily have a continuous extension over $\RR$
(cf. Example~\ref{example:cont:QQ:not:cont:RR}).

Let $C^0(X)$ be the set of real valued continuous functions $f$ on $X(\CC)$ with
$F^*_{\infty}(f) = f$, where
$F_{\infty}  : X(\CC) \to X(\CC)$ is the complex conjugation map on $X(\CC)$.
We denote the group of isomorphism classes of continuous 
hermitian invertible sheaves on $X$ by $\aPic_{C^0}(X)$. 
For details, see Conventions and terminology~\ref{CT:cont:herm:inv:sheaf}.
Here we consider four natural homomorphisms:
\[
\begin{cases}
\overline{\OO} : C^0(X) \to \aPic_{C^0}(X) \quad  (f \mapsto (\OO_X, \exp(-f)\vert\cdot\vert_{can})), \\
\zeta : \aPic_{C^0}(X) \to \Pic(X)  \quad ((L,\vert\cdot\vert) \mapsto L), \\
\mu : C^0(X) \otimes_{\ZZ} \RR \to C^0(X)\quad  (f \otimes x \mapsto  xf ), \\
\overline{\OO} \otimes \operatorname{id}_{\RR} :
C^0(X) \otimes _{\ZZ} \RR \to \aPic_{C^0}(X) \otimes_{\ZZ}  \RR\quad(f \otimes x \mapsto \overline{\OO}(f) \otimes x).
\end{cases}
\]
If we define $\aPic_{C^0}(X)_{\RR}$ to be
\[
\aPic_{C^0}(X)_{\RR} :=\aPic_{C^0}(X) \otimes_{\ZZ}  \RR/(\overline{\OO} \otimes \operatorname{id}_{\RR})(\Ker(\mu)),
\]
then the above homomorphisms yield a commutative diagram
\[
\begin{CD}
C^0(X) @>{\overline{\OO}}>> \aPic_{C^0}(X) @>{\zeta}>> \Pic(X) @>>> 0 \\
@| @VVV @VVV \\
C^0(X) @>{\overline{\OO}}>> \aPic_{C^0}(X)_{\RR} @>{\zeta}>> \Pic(X) \otimes_{\ZZ} \RR @>>> 0
\end{CD}
\]
with exact horizontal sequences.
The purpose of this paper is to prove the following theorem:

\begin{Theorem}
\label{thm:intro:A}
There is a unique positively homogeneous function
$\avol : \aPic_{C^0}(X)_{\RR} \to \RR$ of degree $d$ with the following properties
\rom{(}cf. Theorem~\rom{\ref{thm:cont:extension:aPic}}
\rom{):}
\begin{enumerate}
\renewcommand{\labelenumi}{\rom{(\arabic{enumi})}}
\item
\rom{(}cf. Proposition~\rom{\ref{cor:thm:cont:extension:aPic}}\rom{)}
Let $\{ x_n \}_{n=1}^{\infty}$ be a sequence in a finite dimensional real vector subspace 
$V$ of $\aPic_{C^0}(X)_{\RR}$
and
$\{ f_n \}_{n=1}^{\infty}$ a sequence in $C^0(X)$ such that
$\{ x_n \}_{n=1}^{\infty}$ converges to $x$ in the usual topology of $V$ and
$\{ f_n \}_{n=1}^{\infty}$ converges uniformly to $f$.
Then
\[
\lim_{n\to\infty} \avol\left(x_n
+ \overline{\OO}(f_n)\right)
= \avol\left(x 
+ \overline{\OO}(f)\right).
\]

\item
\rom{(}cf. Theorem~\rom{\ref{thm:vol:limit:arith:var}}\rom{)}
Let $\{ \overline{A}_n \}_{n=1}^{\infty}$ be a sequence in a finitely generated $\ZZ$-submodule 
$M$ of $\Pic_{C^0}(X)$ and
$\{ f_n \}_{n=1}^{\infty}$ a sequence in $C^0(X)$ such that
$\{ \overline{A}_n \otimes 1/n\}_{n=1}^{\infty}$ converges  to $\overline{A}$ in $M \otimes \RR$
in the usual topology and 
$\{ f_n/n \}_{n=1}^{\infty}$ converges uniformly to $f$.
Then
\[
\lim_{n\to\infty} \frac{\ah(\overline{L}_n + \overline{\OO}(f_n)) }{n^d/d!} \\
=
\avol( \pi(\overline{A}) + \overline{\OO}(f)),
\]
where $\pi$ is the canonical homomorphism
$\aPic_{C^0}(X) \otimes_{\ZZ} \RR \to \aPic_{C^0}(X)_{\RR}$
and $\ah$ means the logarithm of the number of small sections
\rom{(}for details, see Conventions and terminology~\rom{\ref{CT:small:sec}}\rom{)}.
\end{enumerate}
\end{Theorem}

The composition $\avol \cdot \pi : \aPic_{C^0}(X) \otimes_{\ZZ} \RR \overset{\pi}{\longrightarrow}
\aPic_{C^0}(X)_{\RR} \overset{\avol}{\longrightarrow} \RR$
gives an affirmative answer to a continuous extension
over $\RR$ of the arithmetic volume function on $\aPic_{C^0}(X) \otimes_{\ZZ} \QQ$.
The last result  is a generalization of Chen's theorem \cite{Chen}.
This also gives an interpretation of the value of $\avol$ in terms of $\ah$.
Namely, if we want to evaluate
$\avol(\pi(\overline{L}_1 \otimes a_1 + \cdots + \overline{L}_r \otimes a_r))$
for $\overline{L}_1, \ldots, \overline{L}_r \in \aPic_{C^0}(X)$ and $a_1, \ldots, a_r \in \RR$,
then (2) of Theorem~\ref{thm:intro:A} says
\[
\avol(\pi(\overline{L}_1 \otimes a_1 + \cdots + \overline{L}_r \otimes a_r)) = 
\lim_{n\to\infty} \frac{\ah([na_1]\overline{L}_1 + \cdots + [na_r]\overline{L}_r)}{n^d/d!}.
\]

The most important tool to establish the continuity of
$\avol : \aPic(X)\otimes \QQ \to \RR$
was the fundamental estimation theorem \cite[Theorem~3.1]{MoCont}.
Unfortunately, it is insufficient to prove a continuous extension
of the arithmetic volume function over $\RR$.
Actually, we need the following generalization of  the fundamental estimation theorem
as a multi-indexed version:

\begin{Theorem}
\label{thm:intro:B}
We assume that $X$ is generically smooth.
Let $\overline{L}_1, \ldots, \overline{L}_r, \overline{A}$ be 
$C^{\infty}$-hermitian invertible sheaves on $X$.
Then there are positive constants
$a_0$, $C$ and $D$ depending only on
$X$ and $\overline{L}_1, \ldots, \overline{L}_r, \overline{A}$ such that
\begin{align*}
& \ah\left(a_1 \overline{L}_1 + \cdots + a_r \overline{L}_r  + (b-c)\overline{A}\right) \\
& \hspace{4em} 
\leq \ah\left(a_1 \overline{L}_1 + \cdots + a_r \overline{L}_r - c\overline{A} \right) \\
& \hspace{8em} + C b \left(\vert a_1 \vert + \cdots + \vert a_r \vert\right)^{d-1} \\
& \hspace{12em} + D \left(\vert a_1 \vert + \cdots + \vert a_r \vert\right)^{d-1} 
\log\left(\vert a_1 \vert + \cdots + \vert a_r \vert\right)
\end{align*}
for all $a_1, \ldots, a_r, b,c \in \ZZ$
with
\[
\vert a_1 \vert + \cdots + \vert a_r \vert \geq b \geq c \geq 0\quad\text{and}\quad
\vert a_1 \vert + \cdots + \vert a_r \vert \geq a_0.
\]
\end{Theorem}

Using the above estimate, we can show the uniform continuity of
$\avol : \aPic_{C^0}(X) \otimes \QQ \to \RR$ in the following sense;
if $\overline{L}_1, \ldots, \overline{L}_r$ are continuous hermitian invertible
sheaves on $X$ and $f : \QQ^r \to \RR$ is a function given by
\[
f(x_1, \ldots, x_r) = \avol(x_1 \overline{L}_1 + \cdots + x_r \overline{L}_r),
\]
then $f$ is uniformly continuous on any bounded set of $\QQ^r$.
This fact gives us a continuous extension of $\avol$ over $\RR$.

This paper is organized as follows: 
In Section~1, we give the proof of the multi-indexed version of the fundamental
estimation. In Section~2, elementary properties of the arithmetic volume function
for continuous hermitian $\QQ$-invertible sheaves are treated.
In Section~3, we consider the uniform continuity of the arithmetic volume function over $\QQ$.
In Section~4, we establish a continuous extension of the arithmetic volume function over $\RR$.
In Section~5, we prove that the arithmetic volume function over $\RR$ can realized
as the limit of $\ah$ of continuous hermitian invertible sheaves.

\bigskip
\renewcommand{\thesubsubsection}{\arabic{subsubsection}}
\renewcommand{\theequation}{CT.\arabic{subsubsection}.\arabic{Claim}}
\subsection*{Conventions and terminology}
\setcounter{subsubsection}{9}
We use the same conventions and terminology as in \cite{MoCont}.
Besides them, we fix the following conventions and terminology for this paper.

\subsubsection{}
Let $S$ be a set and $r$ a positive integer. For $\pmb{x} =(x_1, \ldots, x_r) \in S^r$,
the $i$-th entry $x_i$ of $\pmb{x}$ is denoted by $\pmb{x}(i)$.
We assume that $S$ has an order $\leq$. Then, for $\pmb{x}, \pmb{x}' \in S^r$,
$\pmb{x} \leq \pmb{x}'$ means that $\pmb{x}(i) \leq \pmb{x}'(i)$ for all $i=1, \ldots, r$.

\subsubsection{}
\label{CT:norm:CC}
Let $p$ be a real number with $p \geq 1$.
For $\pmb{x} \in \CC^r$,
we set 
\[
\vert \pmb{x} \vert_p = \left( \vert \pmb{x}(1) \vert^p + \cdots + \vert \pmb{x}(r) \vert^p \right)^{1/p}.
\]
In particular,
$\vert \pmb{x} \vert_1 = \vert \pmb{x}(1) \vert + \cdots + \vert \pmb{x}(r) \vert$ and
$\vert \pmb{x} \vert_2 = \sqrt{ \vert \pmb{x}(1) \vert^2 + \cdots + \vert \pmb{x}(r) \vert^2 }$
(cf. \cite[Conventions and terminology~2]{MoCont}).
Moreover,
we set $\vert \pmb{x} \vert_{\infty} = \max\{\vert \pmb{x}(1) \vert, \ldots,
\vert \pmb{x}(r) \vert \}$.
Note that $\lim_{p\to\infty} \vert \pmb{x} \vert_p = \vert \pmb{x} \vert_{\infty}$.

\subsubsection{}
\label{CT:scaler}
Let $N$ be a $\ZZ$-module and $K$ a field.
Then $N \otimes_{\ZZ} K$ is a $K$-vector space in the natural way.
We denote the $K$-scalar product of $N \otimes_{\ZZ} K$ by $\ \cdot\ $, that is,
\[
a \cdot (x_1 \otimes a_1 + \cdots + x_r \otimes a_r) =
x_1 \otimes a a_1 + \cdots + x_r \otimes a a_r,
\]
where $x_1, \ldots, x_r \in N$ and $a, a_1, \ldots, a_r \in K$.
Note that the kernel of the natural homomorphism $N \to N \otimes_{\ZZ} \QQ$
is the subgroup consisting of torsion elements of $N$.

\subsubsection{}
\label{CT:a:L}
Let $M$ be a module over a ring $R$.
Let $\pmb{a} \in R^r$ and
$\pmb{L} \in M^r$. For simplicity,
we denote
$\pmb{a}(1)\cdot \pmb{L}(1) + \cdots + \pmb{a}(r) \cdot \pmb{L}(r)$
by $\pmb{a} \cdot \pmb{L}$. 

\subsubsection{}
\label{CT:weak:cont}
Let $\mathbb{K}$ be either $\QQ$ or $\RR$.
Let $V$ be a vector space over $\mathbb{K}$ and
$f : V \to \RR$  a function.
Let $d$ be a non-negative real number.
We say $f$ is a {\em positively homogeneous function} of degree $d$ if
$f(\lambda x) = \lambda^d f(x)$ for all $\lambda \in \mathbb{K}_{\geq 0}$ and
$x \in V$.
Moreover, $f$ is said to be {\em weakly continuous} if, for any finite dimensional
vector subspace $W$ of $V$, $\rest{f}{W} : W \to \RR$ is continuous in the
usual topology.

\subsubsection{}
\label{CT:cont:herm:inv:sheaf}
Let $X$ be a $d$-dimensional projective arithmetic variety.
Let $F_{\infty}  : X(\CC) \to X(\CC)$ be the complex conjugation map on $X(\CC)$.
The set of real valued continuous (resp. $C^{\infty}$-) functions $f$ on $X(\CC)$ with
$F^*_{\infty}(f) = f$ is denoted by
$C^0(X)$ (resp. $C^{\infty}(X)$).
A pair $\overline{L} = (L, \vert\cdot\vert)$ of an invertible sheaf $L$ on $X$ and
a continuous hermitian metric $\vert\cdot\vert$ of $L$ is called
a {\em continuous hermitian invertible sheaf} on $X$ if
the hermitian metric is invariant under $F_{\infty}$.
Moreover, if the metric $\vert\cdot\vert$ is $C^{\infty}$,
then $\overline{L}$ is called a {\em $C^{\infty}$-hermitian invertible sheaf} on $X$.
An element of $\aPic_{C^0}(X) \otimes_{\ZZ} \QQ$ (resp. $\aPic(X) \otimes_{\ZZ} \QQ$)
is called a {\em continuous hermitian
$\QQ$-invertible sheaf} (resp.
{\em $C^{\infty}$-hermitian $\QQ$-invertible sheaf}) on $X$.

\subsubsection{}
\label{CT:small:sec}
Let $(L,\Vert\cdot\Vert)$ be a normed $\ZZ$-module, that is,
$M$ is finitely generated $\ZZ$-module and $\Vert\cdot\Vert$ is a norm of $M \otimes_{\ZZ} \RR$.
According to \cite{MoCont}, $\ah(M, \Vert\cdot\Vert)$ is defined by
\[
\ah(M, \Vert\cdot\Vert) := \log \# \{ x \in M \mid \Vert x \otimes 1 \Vert \leq 1 \}.
\]
Let $\overline{L}$ be a continuous hermitian
invertible sheaf on a projective arithmetic variety. For simplicity,
$\ah(H^0(L),\Vert\cdot\Vert_{\sup})$ is often denoted by
$\ah(\overline{L})$.

\subsubsection{}
\label{CT:sub:metric}
Let $M$ be a compact complex manifold and
$\Phi$ a volume form on $M$.
Let $\overline{A} = (A, \vert\cdot\vert_A)$ and
$\overline{B} = (B, \vert \cdot\vert_B)$ be $C^{\infty}$-hermitian
invertible sheaves on $M$.
Let $t$ be a section of $H^0(M, B)$ such that
$t$ is non-zero on each connected component of $M$.
The subnorm 
induced by an injective homomorphism
$H^0(M, A-B) \overset{\otimes t}{\longrightarrow} H^0(M, A)$
and the natural $L^2$-norm of $H^0(M,A)$ given by
$\Phi$ and $\vert\cdot\vert_A$
is denoted by $\Vert\cdot\Vert^{\overline{A}, A - B}_{L^2, t, \sub}$, that is,
\[
\Vert s \Vert^{\overline{A}, A - B}_{L^2, t, \sub} = \sqrt{\int_{M} \vert s \otimes t \vert^2_A \Phi}
\]
for $s \in H^0(M, A-B)$.
For simplicity,
$\Vert\cdot\Vert^{\overline{A}, A - B}_{L^2, t, \sub}$ is often
denoted by $\Vert\cdot\Vert^{\overline{A}}_{L^2, t, \sub}$.

\renewcommand{\theTheorem}{\arabic{section}.\arabic{subsection}.\arabic{Theorem}}
\renewcommand{\theClaim}{\arabic{section}.\arabic{subsection}.\arabic{Theorem}.\arabic{Claim}}
\renewcommand{\theequation}{\arabic{section}.\arabic{subsection}.\arabic{Theorem}.\arabic{Claim}}

\section{A multi-indexed version of  the fundamental estimation}

\subsection{}
Let $X$ be a $d$-dimensional generically smooth 
projective arithmetic variety.
Let $\overline{A}$ be a $C^{\infty}$-hermitian invertible sheaf on $X$ and
$\overline{\pmb{L}} = (\overline{L}_1, \ldots, \overline{L}_r)$
a finite sequence of
$C^{\infty}$-hermitian invertible sheaves on $X$.
Let $\pmb{L} = (L_1, \ldots, L_r)$ be the sequence of
invertible sheaves  obtained by forgetting metrics of
$\overline{\pmb{L}}$.
The following theorem is a generalization of  \cite[Theorem~3.1]{MoCont}.

\begin{Theorem}
\label{thm:h:0:estimate:big:main}
There are positive constants
$a_0$, $C$ and $D$ depending only on
$X$, $\overline{\pmb{L}}$ and
$\overline{A}$ such that
\begin{multline*}
\ah\left(H^0(\pmb{a} \cdot \pmb{L} + (b-c)A),
\Vert\cdot\Vert^{\pmb{a} \cdot \overline{\pmb{L}} + (b-c) \overline{A}}_{\sup}\right) \leq
\ah\left(H^0(\pmb{a}\cdot \pmb{L} - cA), 
\Vert\cdot\Vert^{\pmb{a} \cdot \overline{\pmb{L}} - c \overline{A}}_{\sup}\right) \\
+ C b \vert \pmb{a} \vert_{1}^{d-1} + D \vert \pmb{a} \vert_{1}^{d-1} \log(\vert \pmb{a} \vert_{1})
\end{multline*}
for all $\pmb{a} \in \ZZ^r$ and $b,c \in \ZZ$
with $\vert \pmb{a} \vert_{1} \geq b \geq c \geq 0$ and
$\vert \pmb{a} \vert_{1} \geq a_0$, where
\[
\begin{cases}
\vert \pmb{a} \vert_1 = \vert a_1 \vert + \cdots + \vert a_r \vert, \\
\pmb{a} \cdot \pmb{L} = a_1 L_1 + \cdots + a_r L_r, \\
\pmb{a} \cdot \overline{\pmb{L}} = a_1 \overline{L}_1 + \cdots + a_r \overline{L}_r
\end{cases}
\]
for $\pmb{a} = (a_1, \ldots, a_r) \in \ZZ^r$ \rom{(}cf.
Conventions and terminology~\rom{\ref{CT:norm:CC}}
and \rom{\ref{CT:a:L}}\rom{)}.
\end{Theorem}

The proof of Theorem~\ref{thm:h:0:estimate:big:main} is almost same as
one of  \cite[Theorem~3.1]{MoCont}. For reader's convenience, we will give its proof in
the remaining of this section. We use the same notation as in \cite{MoCont}.
Let us begin with distorsion functions.

\subsection{Distorsion function}
\setcounter{Theorem}{0}
Let $M$ be an $n$-equidimensional projective complex manifold.
First let us recall distorsion functions.
Let  $\Phi$ be a volume form of $M$ and
$\overline{H} = (H, h)$ a $C^{\infty}$-hermitian invertible
sheaf on $M$.
For $s, s' \in H^0(M, H)$, we set
\[
\langle s, s' \rangle_{\overline{H}, \Phi} =
\int_M h(s, s') \Phi.
\]
Let $s_1, \ldots, s_N$ be an orthonormal basis of $H^0(X, H)$
with respect to $\langle\ , \ \rangle_{\overline{H}, \Phi}$.
Then it is easy to see that, for all $x \in M$,
the quantity $\sum_{i=1}^N h(s_i, s_i)(x)$ does not depend on the choice of
the orthonormal basis $s_1, \ldots, s_N$, so that
we define
\[
\dist(\overline{H}, \Phi)(x) =
\sum_{i=1}^N h(s_i, s_i)(x).
\]
The function $\dist(\overline{H}, \Phi)$ is called
the {\em distorsion function of $\overline{H}$ with respect to $\Phi$}.
For a positive number $\lambda$, it is easy to check that
$\dist(\overline{H}, \lambda \Phi) = \lambda^{-1} \dist(\overline{H}, \Phi)$.
Moreover, if $M_1, \ldots, M_l$ are connected components of $M$, then
\[
\dist(\overline{H}, \Phi)  = \dist\left(\rest{\overline{H}}{X_1}, \rest{\Phi}{X_1}\right)  + \cdots
+ \dist\left(\rest{\overline{H}}{X_l}, \rest{\Phi}{X_l}\right).
\]

Let $\overline{A}$ be a positive  $C^{\infty}$-hermitian invertible sheaf on $M$
and  $\overline{\pmb{B}} = (\overline{B}_1,
\ldots, \overline{B}_l)$  a finite sequence of 
positive $C^{\infty}$-hermitian invertible sheaves
on $M$. Then we have the following:

\begin{Theorem}
\label{thm:estimate:disfun:A:B}
For any real number $\epsilon$ with $0< \epsilon < 1$,
there is a positive constant
$a(\epsilon)$ such that
\[
\dist(a \overline{A} - \pmb{b} \cdot \overline{\pmb{B}}, 
c_1(\overline{A})^n)
\leq \frac{(1+\epsilon)^{l+1}}{(1-\epsilon)^l} h^0(a A) 
\]
for all $a \in \ZZ$ and
$\pmb{b} = (b_1, \ldots, b_l) \in \ZZ^l$ with $a \geq a(\epsilon),
b_1 \geq a(\epsilon), \ldots, b_l \geq a(\epsilon)$, where
$\pmb{b} \cdot \overline{\pmb{B}} = b_1 \overline{B}_1 + \cdots + b_l \overline{B}_l$.
\end{Theorem}

\begin{proof}
Clearly we may assume that $M$ is connected.
In the following, the hermitian metrics of $a\overline{A}$, $b_i \overline{B}_i$ and
$\pmb{b} \cdot \overline{\pmb{B}}$ are denoted by
$h_{a\overline{A}}$, $h_{b_i \overline{B}_i}$ and
$h_{\pmb{b} \cdot \overline{\pmb{B}}}$ respectively.
By Bouche-Tian's theorem (\cite{Bou}, \cite{Tian}), there is a positive constant
$a(\epsilon)$ such that
\[
\begin{cases}
h^0(aA)\left(1- \epsilon \right) \leq \dist(a \overline{A}, \Phi(\overline{A}))(z) \leq 
h^0(aA) \left( 1 + \epsilon \right) \\
h^0(b_i B_i)\left( 1-\epsilon \right)
\leq \dist(b_i \overline{B}_i, \Phi(\overline{B}_i))(z) \leq h^0(b_i B_i) \left( 1 + \epsilon \right)
\quad(i=1, \ldots, l)
\end{cases}
\]
hold for all
$z \in M$ and all $a \geq a(\epsilon), b_1 \geq a(\epsilon), \ldots, b_l \geq a(\epsilon)$,
where 
\[
\Phi(\overline{A}) = \frac{c_1(\overline{A})^n}{\int_M c_1(\overline{A})^n}\quad\text{and}\quad
\Phi(\overline{B}_i) = \frac{c_1(\overline{B}_i)^n}{\int_M c_1(\overline{B}_i)^n}
\]
for $i=1, \ldots, l$.

Fix an arbitrary point $x$ of $M$.
Let
$\pmb{b} = (b_1, \ldots, b_l) \in \ZZ^l$ with
$b_1 \geq a(\epsilon), \ldots, b_l \geq a(\epsilon)$.
We consider an orthonormal basis of
$H^0(b_iB_i)$ with respect to 
the $L^2$-norm $\langle\ , \ \rangle_{b_i \overline{B}_i, \Phi(\overline{B}_i)}$
arising from $h_{b_i\overline{B}_i}$ and $\Phi(\overline{B}_i)$
such that the only one element of the basis has non-zero value
at $x$.
Let $s_{b_i}$ be a such element of the basis.
Then we have
\[
h_{b_i\overline{B}_i}(s_{b_i}, s_{b_i})(x) = 
\dist(b_i\overline{B}_i, \Phi(\overline{B}_i))(x) 
\geq (1-\epsilon) h^0(b_iB_i).
\]
On the other hand, since
\[
\Vert s_{b_i} \Vert^2_{\sup} \leq 
\sup_{z \in M} \dist(b_i\overline{B}_i, \Phi(\overline{B}_i))(z) \leq
(1+\epsilon) h^0(b_iB_i),
\]
we obtain
\[
\frac{h_{b_i\overline{B}_i}(s_{b_i}, s_{b_i})(x)}{\Vert s_{b_i} \Vert^2_{\sup}} \geq
\frac{1-\epsilon}{1+\epsilon}.
\]
If we set
$s_{\pmb{b}} = s_{b_1} \otimes \cdots \otimes s_{b_l}$, then
\[
\frac{h_{\pmb{b} \cdot \overline{\pmb{B}}}(s_{\pmb{b}}, s_{\pmb{b}})(x)}
{\Vert s _{\pmb{b}} \Vert_{\sup}^2} \geq
\frac{h_{b_1\overline{B}_1}(s_{b_1}, s_{b_1})(x) \cdots h_{b_l\overline{B}_l}(s_{b_l}, s_{b_l})(x)}{\Vert s_{b_1} \Vert_{\sup}^2 \cdots \Vert s_{b_l} \Vert_{\sup}^2}
\geq \left(\frac{1-\epsilon}{1+\epsilon}\right)^l.
\]

For $a \geq a(\epsilon)$, let $t_1, \ldots, t_r$ be an orthonormal basis of
$H^0(aA - \pmb{b} \cdot \overline{\pmb{B}})$ with respect to
$\langle\ ,\ \rangle_{a\overline{A}- \pmb{b} \cdot \overline{\pmb{B}}, \Phi(\overline{A})}$
such that $s_{\pmb{b}} \otimes t_1, \ldots s_{\pmb{b}} \otimes t_r$
are orthogonal with respect to
$\langle\ ,\ \rangle_{a\overline{A}, \Phi(\overline{A})}$
as elements of $H^0(aA)$.
Then, since
\[
\left\{ s_{\pmb{b}} \otimes t_i/\Vert s_{\pmb{b}} \otimes t_i \Vert_{a\overline{A}, \Phi(\overline{A})} \right\}_{i=1, \ldots, r}
\]
form a part of an orthonormal basis of $H^0(aA)$,
\[
\sum_{i=1}^r \frac{h_{a \overline{A}}(s_{\pmb{b}} \otimes t_i, s_{\pmb{b}} \otimes t_i)(x)}
{\Vert s_{\pmb{b}} \otimes t_i \Vert^2_{a\overline{A}, \Phi(\overline{A})}} \leq 
\dist(aA, \Phi(\overline{A}))(x) \leq (1 + \epsilon) h^0(aA).
\]
Note that $\Vert s_{\pmb{b}} \otimes t_i \Vert^2_{a\overline{A}, \Phi(\overline{A})} \leq
\Vert s_{\pmb{b}} \Vert^2_{\sup}$.
Moreover, since $\lambda = \int_M c_1(\overline{A})^n \geq 1$,
\[
\dist(a\overline{A} - \pmb{b} \cdot \overline{\pmb{B}}, 
c_1(\overline{A})^n) = \lambda^{-1} \dist(a\overline{A} - \pmb{b} \cdot \overline{\pmb{B}}, 
\Phi(\overline{A})) \leq \dist(a\overline{A} - \pmb{b} \cdot \overline{\pmb{B}}, 
\Phi(\overline{A})).
\]
Therefore,
\begin{multline*}
 \left(\frac{1-\epsilon}{1+\epsilon}\right)^l \dist(a\overline{A} - \pmb{b} \cdot \overline{\pmb{B}}, 
c_1(\overline{A})^n)(x) \\
\leq
\frac{h_{\pmb{b} \cdot \overline{\pmb{B}}}(s_{\pmb{b}}, s_{\pmb{b}})(x)}{\Vert s_{\pmb{b}} \Vert^2_{\sup}} 
\sum_{i=1}^r h_{a\overline{A}-\pmb{b} \cdot \overline{\pmb{B}}}(t_i,  t_i)(x) \\
\qquad\qquad\qquad \leq
\sum_{i=1}^r  \frac{h_{\pmb{b} \cdot \overline{\pmb{B}}}(s_{\pmb{b}}, s_{\pmb{b}})(x)}{\Vert s_{\pmb{b}} \otimes t_i \Vert^2_{a\overline{A}, \Phi(\overline{A})}} 
h_{a\overline{A}-\pmb{b} \cdot \overline{\pmb{B}}}(t_i,  t_i)(x) \\
= \sum_{i=1}^r \frac{h_{a \overline{A}}(s_{\pmb{b}} \otimes t_i, s_{\pmb{b}} \otimes t_i)(x)}
{\Vert s_{\pmb{b}} \otimes t_i \Vert^2_{a\overline{A}, \Phi(\overline{A})}} \leq
(1 + \epsilon)h^0(aA).
\end{multline*}
Thus the theorem follows.
\end{proof}

Here we recall several notations:
Let $\KK$ be either $\RR$ or $\CC$.
Let $V$ be a finite dimensional vector space over $\KK$.
A map $\langle\ ,\ \rangle : V \times V \to \KK$ is called
a {\em $\KK$-inner product} if the following conditions (1) $\sim$ (4) are satisfied:
(1) $\langle x, y \rangle = \overline{\langle y, x \rangle}$
($\forall x, y \in V$),
(2) $\langle x + x', y \rangle = \langle x, y \rangle + \langle x' , y \rangle$ and
$\langle a x, y \rangle = a \langle x, y \rangle$
($\forall x,x',y \in V,  \forall a \in \KK$),  (3) $\langle x, x \rangle \geq 0$ ($\forall x \in V$), 
(4) $\langle x, x \rangle = 0$ $\Longleftrightarrow$ $x = 0$.
Let $(V_1, \langle\ ,\ \rangle_1)$ and
$(V_2, \langle\ ,\ \rangle_2)$ be finite dimensional vector spaces over $\KK$ with
$\KK$-inner products $\langle\ ,\ \rangle_1$ and $\langle\ ,\ \rangle_2$, and
let $\phi : V_1 \to V_2$ be an isomorphism over $\KK$.
For a basis $\{ x_1, \ldots, x_n \}$ of $V_1$, we consider
a quantity
\[
-\frac{1}{2} \log \left( \frac{\det (\langle x_i, x_j \rangle_1)}{\det (\langle \phi(x_i), \phi(x_j) \rangle_2)} \right),
\]
which
does not depend on the choice of the basis $\{ x_1, \ldots, x_n \}$ of $V_1$.
It is called the {\em volume difference} of
$(V_1, \langle\ ,\ \rangle_1) \overset{\phi}{\longrightarrow}(V_2, \langle\ ,\ \rangle_2)$
and is denoted by
\[
\voldiff((V_1, \langle\ ,\ \rangle_1) \overset{\phi}{\longrightarrow}(V_2, \langle\ ,\ \rangle_2)).
\]
It is easy to check that
\[
[\KK:\RR]  \voldiff((V_1, \langle\ ,\ \rangle_1) \overset{\phi}{\longrightarrow}(V_2, \langle\ ,\ \rangle_2)) =
\log \left( \frac{\vol \{ x \in V_1 \mid \langle x, x \rangle_1 \leq 1 \}}%
{\vol \{ x \in V_1 \mid \langle \phi(x), \phi(x) \rangle_2 \leq 1 \}} \right),
\]
where $\vol$ is a Haar measure of $V_1$.
Thus if $(M, \Vert\cdot\Vert_1)$ and
$(M, \Vert\cdot\Vert_2)$ are normed $\ZZ$-modules and
$\Vert\cdot\Vert_1$ and $\Vert\cdot\Vert_2$ are $L^2$-norms,
then
\[
\achi(M, \Vert\cdot\Vert_1) - \achi(M, \Vert\cdot\Vert_2) = 
\voldiff((M \otimes_{\ZZ} {\RR}, \Vert\cdot\Vert_1) \overset{\operatorname{id}}{\longrightarrow}(M \otimes_{\ZZ} {\RR}, \Vert\cdot\Vert_2))
\]

\bigskip
Let us consider a corollary of Theorem~\ref{thm:estimate:disfun:A:B}.
Let 
$\overline{L}_1, \ldots, \overline{L}_r, \overline{A}$
be $C^{\infty}$-hermitian invertible sheaves
on $M$ such that
$\overline{A}$ and $\overline{L}_i + \overline{A}$ are positive
for $i=1, \ldots, r$.
We set
\[
\pmb{L} = (L_1, \ldots, L_r),\quad
\overline{\pmb{L}} = (\overline{L}_1, \ldots, \overline{L}_r),\quad
\Phi = c_1(\overline{L}_1 + \cdots + \overline{L}_r + r \overline{A})^n.
\] 
Let $\pmb{a}  \in \ZZ_{\geq 0}^r$ and $b, c \in \ZZ_{\geq 0}$.
Let $s$ be a section of $H^0(bA)$ such that
$\Vert s \Vert_{\sup} \leq 1$ and $s$ is non-zero on each connected component of $M$. Let
\[
\langle\ ,\ \rangle_{\pmb{a} \cdot \overline{\pmb{L}}- c \overline{A}}\quad\text{and}\quad
\langle\ ,\ \rangle_{\pmb{a} \cdot \overline{\pmb{L}} + (b-c) \overline{A}}
\]
be the natural $\CC$-inner products of $H^0(\pmb{a} \cdot \pmb{L} - cA)$ and
$H^0(\pmb{a} \cdot \pmb{L}+ (b-c)A)$ with respect to
$\Phi$. Here we consider a submetric
$\langle\ ,\ \rangle_{\pmb{a} \cdot \overline{\pmb{L}} + (b-c) \overline{A},s,\sub}$ of $H^0(\pmb{a} \cdot \pmb{L}- cA)$
induced by
$s$ and the metric
$\langle\ ,\ \rangle_{\pmb{a} \cdot \overline{\pmb{L}} + (b-c) \overline{A}}$, that is,
\[
\langle t,t'  \rangle_{\pmb{a} \cdot \overline{\pmb{L}} + (b-c) \overline{A},s,\sub}
= \langle s t, st' \rangle_{\pmb{a} \cdot \overline{\pmb{L}}+ (b-c) \overline{A}}
\]
for $t, t' \in H^0(\pmb{a} \cdot \pmb{L} - cA)$.
Then we have the following corollary.

\begin{Corollary}
\label{cor:comp:sub:L2:ball}
Let $\gamma(\pmb{a},c,s)$ be the volume difference of
\[
\left( H^0(\pmb{a} \cdot \pmb{L}- cA), \langle ,\rangle_{\pmb{a} \cdot \overline{\pmb{L}} - c \overline{A}}
\right)
\overset{\operatorname{id}}{\longrightarrow} \left( H^0(\pmb{a} \cdot \pmb{L} - cA), \langle , \rangle_{\pmb{a} \cdot \overline{\pmb{L}} + (b-c) \overline{A},s,\sub} \right).
\]
For any real number $\epsilon$ with $0 < \epsilon < 1$,
there is positive constant $a(\epsilon)$ such that
\[
\gamma(\pmb{a}, c,s) \geq \frac{(1+\epsilon)^{r+2}}{(1-\epsilon)^{r+1}} h^0(2\vert \pmb{a} \vert_1(L_1 + \cdots + L_r + rA)) \left(\int_M \log(\vert s \vert) \Phi \right)
\]
for all $\pmb{a}  \in \ZZ_{\geq 0}^r$ and
$c \in \ZZ_{\geq 0}$ with
$\vert \pmb{a} \vert_1 \geq a(\epsilon)$.
\end{Corollary}

\begin{proof}
Note that, for all $\pmb{a} = (a_1, \ldots, a_r) \in \ZZ_{\geq 0}^r$ and $c \in \ZZ_{\geq 0}$,
\begin{multline*}
\pmb{a} \cdot \overline{\pmb{L}} - c \overline{A} = 
2\vert \pmb{a} \vert_1 (\overline{L}_1 + \cdots + \overline{L}_r + r\overline{A}) \\
- \left(2\vert \pmb{a} \vert_1-a_1\right) (\overline{L}_1 + \overline{A})
- \cdots 
- \left(2\vert \pmb{a} \vert_1-a_r\right) (\overline{L}_r + \overline{A})
- (c + \vert \pmb{a} \vert_1) \overline{A}
\end{multline*}
and that $2\vert \pmb{a} \vert_1- a_i \geq \vert \pmb{a} \vert_1$ for all $i$.
Therefore,
by Theorem~\ref{thm:estimate:disfun:A:B}, there is a positive
constant $a(\epsilon)$
such that
\[
\dist(\pmb{a} \cdot \overline{\pmb{L}} -c \overline{A}, \Phi)
\leq \frac{(1+\epsilon)^{r+2}}{(1-\epsilon)^{r+1}} h^0(2\vert \pmb{a} \vert_1(L_1+\cdots + L_r+rA))
\]
for all $x \in M$ and all $\pmb{a}  \in \ZZ_{\geq 0}^r$ and
$c \in \ZZ_{\geq 0}$ with
$\vert \pmb{a} \vert_1 \geq a(\epsilon)$.
Let $t_1, \ldots, t_l$ be an orthonormal basis
of $H^0(\pmb{a} \cdot \pmb{L} - cA)$ with respect to
$\langle\ ,\ \rangle_{\pmb{a} \cdot \overline{\pmb{L}} - c \overline{A}}$
such that
$st_1, \ldots, st_l$ are orthogonal with respect to 
$\langle\ ,\ \rangle_{\pmb{a} \cdot \overline{\pmb{L}} + (b-c) \overline{A}}$.
Then
\[
\gamma(\pmb{a}, c,s) = -\frac{1}{2} \log \left(  
\frac{\det ( \langle t_i, t_j \rangle_{\pmb{a} \cdot \overline{\pmb{L}} - c \overline{A}})}{\det(\langle st_i ,st_j \rangle_{\pmb{a} \cdot \overline{\pmb{L}} - c \overline{A}})}
\right) = \frac{1}{2} \sum_{i=1}^l \log \int_M \vert s \vert^2 \vert t_i \vert^2 \Phi.
\]
Thus, using Jensen's inequality,
for all $\pmb{a}  \in \ZZ_{\geq 0}^r$ and $c \in \ZZ_{\geq 0}$
with $\vert \pmb{a} \vert_1 \geq a(\epsilon)$,
{\allowdisplaybreaks
\begin{align*}
\gamma(\pmb{a}, c, s) & \geq  \frac{1}{2} \sum_{i=1}^l  \int_M \log(\vert s \vert^2) \vert t_i \vert^2 \Phi \\
& =   \int_M \log(\vert s \vert) \dist(\pmb{a} \cdot \overline{\pmb{L}} - c \overline{A}, \Phi) \Phi \\
& \geq \frac{(1+\epsilon)^{r+2}}{(1-\epsilon)^{r+1}} h^0(2\vert \pmb{a} \vert_1(L_1 + \cdots + L_r+rA)) \left(\int_M \log(\vert s \vert)\Phi\right).
\end{align*}}
\end{proof}

\subsection{The proof of Theorem~\ref{thm:h:0:estimate:big:main}}
\setcounter{Theorem}{0}
In this subsection, let us give the proof of  Theorem~\ref{thm:h:0:estimate:big:main}.
Here we consider variants of Theorem~\ref{thm:h:0:estimate:big:main},
that is, the restricted version of Theorem~\ref{thm:h:0:estimate:big:main}
and the $L^2$-version of Theorem~\ref{thm:h:0:estimate:big:>=:0}.

\begin{Theorem}
\label{thm:h:0:estimate:big:>=:0}
In the situation of Theorem~\rom{\ref{thm:h:0:estimate:big:main}},
there are positive constants
$a_0$, $C$ and $D$ depending only on
$X$, $\overline{\pmb{L}}$ and
$\overline{A}$ such that
\begin{multline*}
\ah\left(H^0(\pmb{a} \cdot \pmb{L} + (b-c)A),
\Vert\cdot\Vert^{\pmb{a} \cdot \overline{\pmb{L}} + (b-c) \overline{A}}_{\sup}\right) \leq
\ah\left(H^0(\pmb{a}\cdot \pmb{L} - cA), 
\Vert\cdot\Vert^{\pmb{a} \cdot \overline{\pmb{L}} - c \overline{A}}_{\sup}\right) \\
+ C b \vert \pmb{a} \vert_{1}^{d-1} + D \vert \pmb{a} \vert_{1}^{d-1} \log(\vert \pmb{a} \vert_{1})
\end{multline*}
for all $\pmb{a} \in \ZZ_{\geq 0}^r$ and $b,c \in \ZZ$
with $\vert \pmb{a} \vert_{1} \geq b \geq c \geq 0$ and
$\vert \pmb{a} \vert_{1} \geq a_0$.
\end{Theorem}

\begin{Theorem}
\label{thm:h:0:estimate:big:main:L:2}
In the situation of Theorem~\rom{\ref{thm:h:0:estimate:big:main}},
we fix a volume form of $X(\CC)$ to give
$L^2$-norms $\Vert\cdot\Vert^{\pmb{a} \cdot \overline{\pmb{L}} + (b-c) \overline{A}}_{L^2}$ and $\Vert\cdot\Vert^{\pmb{a}\cdot \overline{\pmb{L}} - c \overline{A}}_{L^2}$.
Then
there are positive constants
$a'_0$, $C'$ and $D'$ depending only on
$X$, $\overline{\pmb{L}}$,
$\overline{A}$ and the volume form of $X(\CC)$ such that
\begin{multline*}
\ah\left(H^0(\pmb{a} \cdot \pmb{L} + (b-c)A),
\Vert\cdot\Vert^{\pmb{a} \cdot \overline{\pmb{L}} + (b-c) \overline{A}}_{L^2}\right) \leq
\ah\left(H^0(\pmb{a} \cdot \pmb{L} - cA), 
\Vert\cdot\Vert^{\pmb{a}\cdot \overline{\pmb{L}} - c \overline{A}}_{L^2}\right) \\
+ C' b \vert \pmb{a} \vert_{1}^{d-1} + D' \vert \pmb{a} \vert_{1}^{d-1} \log(\vert \pmb{a} \vert_{1})
\end{multline*}
for all $\pmb{a} \in \ZZ_{\geq 0}^r$ and $b,c \in \ZZ$
with $\vert \pmb{a} \vert_{1} \geq b \geq c \geq 0$ and
$\vert \pmb{a} \vert_{1} \geq a'_0$.
\end{Theorem}

First of all,
let us see 
\[
\begin{cases}
\text{Theorem~\ref{thm:h:0:estimate:big:>=:0} $\Longrightarrow$
Theorem~\ref{thm:h:0:estimate:big:main}}, \\
\text{Theorem~\ref{thm:h:0:estimate:big:main:L:2} $\Longrightarrow$
Theorem~\ref{thm:h:0:estimate:big:>=:0}},
\end{cases}
\]
so that it is sufficient to show Theorem~\ref{thm:h:0:estimate:big:main:L:2}.

\begin{proof}[Theorem~\ref{thm:h:0:estimate:big:>=:0} $\Longrightarrow$ Theorem~\ref{thm:h:0:estimate:big:main}]
For $\pmb{\epsilon} \in \{ \pm 1 \}^r$ and $\pmb{a} \in \ZZ^r$,
we set 
\[
\overline{\pmb{L}}(\pmb{\epsilon}) =
(\pmb{\epsilon}(1) \overline{L}_1, \ldots, \pmb{\epsilon}(r) \overline{L}_r)\quad\text{and}\quad
\pmb{a}(\pmb{\epsilon}) = (\pmb{\epsilon}(1) \pmb{a}(1), \ldots, \pmb{\epsilon}(r) \pmb{a}(r)).
\]
By Theorem~\ref{thm:h:0:estimate:big:>=:0}, for each $\pmb{\epsilon} \in \{ \pm 1 \}^r$,
there are positive constants
$a_0(\pmb{\epsilon})$, $C(\pmb{\epsilon})$ and $D(\pmb{\epsilon})$ depending only on
$X$, $\overline{\pmb{L}}(\pmb{\epsilon})$ and
$\overline{A}$ such that
\begin{multline*}
\ah\left(H^0(\pmb{a} \cdot \pmb{L}(\pmb{\epsilon}) + (b-c)A),
\Vert\cdot\Vert^{\pmb{a} \cdot \overline{\pmb{L}}(\pmb{\epsilon}) + (b-c) \overline{A}}_{\sup}\right) \\
\leq
\ah\left(H^0(\pmb{a}\cdot \pmb{L}(\pmb{\epsilon}) - cA), 
\Vert\cdot\Vert^{\pmb{a} \cdot \overline{\pmb{L}}(\pmb{\epsilon}) - c \overline{A}}_{\sup}\right) \\
+ C(\pmb{\epsilon}) b \vert \pmb{a} \vert_{1}^{d-1} + D(\pmb{\epsilon}) \vert \pmb{a} \vert_{1}^{d-1} \log(\vert \pmb{a} \vert_{1})
\end{multline*}
for all $\pmb{a} \in \ZZ_{\geq 0}^r$ and $b,c \in \ZZ$
with $\vert \pmb{a} \vert_{1} \geq b \geq c \geq 0$ and
$\vert \pmb{a} \vert_{1} \geq a_0(\pmb{\epsilon})$.
Note that, for any $\pmb{a} \in \ZZ^r$, there is $\pmb{\epsilon} \in \{ \pm 1 \}^r$
with $\pmb{a}(\pmb{\epsilon}) \in \ZZ_{\geq 0}^r$, and that
$\pmb{a}(\pmb{\epsilon}) \cdot \overline{\pmb{L}}(\pmb{\epsilon}) =
\pmb{a} \cdot \overline{\pmb{L}}$ and $\vert \pmb{a}(\pmb{\epsilon}) \vert_1 =
\vert \pmb{a} \vert_1$ for $\pmb{\epsilon} \in \{ \pm 1 \}^r$ and $\pmb{a} \in \ZZ^r$.
Thus, if we set 
\[
a_0 = \max_{\pmb{\epsilon} \in \{\pm 1 \}^r} \{ a_0(\pmb{\epsilon}) \},\quad
C  = \max_{\pmb{\epsilon} \in \{\pm 1 \}^r} \{ C(\pmb{\epsilon}) \}\quad\text{and}\quad
D = \max_{\pmb{\epsilon} \in \{\pm 1 \}^r} \{ D(\pmb{\epsilon}) \},
\]
then Theorem~~\ref{thm:h:0:estimate:big:main} follows.
\end{proof}

\begin{proof}[Theorem~\ref{thm:h:0:estimate:big:main:L:2} $\Longrightarrow$ Theorem~\ref{thm:h:0:estimate:big:>=:0}]
Since
$\Vert\cdot\Vert^{\pmb{a}\cdot\overline{\pmb{L}} + (b-c)\overline{A}}_{\sup}
\geq \Vert\cdot\Vert^{\pmb{a}\cdot\overline{\pmb{L}} + (b-c)\overline{A}}_{L^2}$, we have
\begin{multline*}
\ah\left(H^0(\pmb{a}\cdot\pmb{L} + (b-c)A),
\Vert\cdot\Vert^{\pmb{a}\cdot\overline{\pmb{L}} + (b-c)\overline{A}}_{\sup}\right) \\
\leq
\ah\left(H^0(\pmb{a}\cdot\pmb{L} + (b-c)A),
\Vert\cdot\Vert^{\pmb{a}\cdot\overline{\pmb{L}} + (b-c)\overline{A}}_{L^2}\right).
\end{multline*}
Moreover, applying Gromov's inequality (cf. \cite[Corollary~1.1.2]{MoCont}) to
$\overline{L}_{1\CC},\ldots,\overline{L}_{r\CC}, -\overline{A}_{\CC}$,
there is a constant $D \geq 1$ such that
\[
\Vert\cdot\Vert^{\pmb{a}\cdot \overline{\pmb{L}} - c \overline{A}}_{L^2} \geq
D^{-1} (\vert \pmb{a}\vert_{1}+c+1)^{-(d-1)} \Vert\cdot\Vert^{\pmb{a}\cdot \overline{\pmb{L}} - c \overline{A}}_{\sup}
\]
for all $\pmb{a} \in \ZZ_{\geq 0}^r$ and $c \in \ZZ_{\geq 0}$.
Therefore, since $\vert \pmb{a} \vert_{1} \geq c$, by using
\cite[Proposition~2,1]{MoCont}, we obtain
\begin{multline*}
\ah\left(H^0(\pmb{a}\cdot \pmb{L} - c A), 
\Vert\cdot\Vert^{\pmb{a}\cdot \overline{\pmb{L}} - c \overline{A}}_{L^2}\right) \\
\leq \ah\left(H^0(\pmb{a}\cdot \pmb{L} - c A), 
D^{-1} (\vert \pmb{a} \vert_{1}+c+1)^{-(d-1)} \Vert\cdot\Vert^{\pmb{a}\cdot \overline{\pmb{L}} - c \overline{A}}_{\sup}\right) \\
\leq \ah\left(H^0(\pmb{a}\cdot \pmb{L} - c A), 
\Vert\cdot\Vert^{\pmb{a}\cdot \overline{\pmb{L}} - c \overline{A}}_{\sup}\right) \\
+  \log(D(2\vert \pmb{a} \vert_{1}+ 1)^{d-1}) C_1 \vert \pmb{a} \vert_{1}^{d-1} + C_2 \vert \pmb{a}\vert_{1}^{d-1} \log(\vert \pmb{a} \vert_{1}),
\end{multline*}
where $C_1$ and $C_2$ are positive constants with the following properties:
\addtocounter{Claim}{1}
\begin{equation}
\label{eqn:L2:imply:sup}
\begin{cases}
\rank H^0(a_1(L_1+A) + \cdots + a_r(L_r+A)) 
\leq C_1 \vert \pmb{a} \vert_{1}^{d-1} \quad(\vert \pmb{a} \vert_{1} \geq 1), \\
\log(18) \rank H^0(a_1(L+A)+\cdots + a_r(L_r+A)) \\
\qquad + 2 \log\left( (\rank H^0(a_1(L_1+A)+ \cdots + a_r(L_r + A)))! \right) \\
\phantom{\log(18) \rank H^0(a_1(L_1+A)+\cdots +}
\leq C_2 \vert \pmb{a} \vert_{1}^{d-1} \log(\vert \pmb{a} \vert_{1})
\quad(\vert \pmb{a} \vert_{1} \geq 2).
\end{cases}
\end{equation}
Thus we get our assertion.
\end{proof}

\begin{proof}[The proof of Theorem~\ref{thm:h:0:estimate:big:main:L:2}]
First let us see the following claim.

\begin{Claim}
\label{claim:thm:h:0:estimate:big:1-1}
Let $\overline{A}'$ be another $C^{\infty}$-hermitian
invertible sheaf on $X$ with
$\overline{A} \leq \overline{A}'$.
If the theorem holds for
$\overline{\pmb{L}}$ and $\overline{A}'$,
then it also holds for $\overline{\pmb{L}}$ and $\overline{A}$.
\end{Claim}

\begin{proof}
This is obvious because
$\pmb{a} \cdot \overline{\pmb{L}} + (b-c) \overline{A} \leq 
\pmb{a} \cdot \overline{\pmb{L}} + (b-c) \overline{A}'$ and
$\pmb{a} \cdot \overline{\pmb{L}} - c \overline{A}' \leq 
\pmb{a}\cdot \overline{\pmb{L}} - c \overline{A}$.
\end{proof}

By the above claim, we may assume the following:

\begin{enumerate}
\renewcommand{\labelenumi}{(\arabic{enumi})}
\item
$A$ is very ample on $X$.

\item
$\overline{A}$ and $\overline{L}_i + \overline{A}$
($i=1, \ldots, r$) are positive on $X(\CC)$.
\end{enumerate}
Moreover, we fix positive constants $C_1$ and $C_2$ as in \eqref{eqn:L2:imply:sup}.

\begin{Claim}
\label{claim:thm:h:0:estimate:big:1-2}
If the theorem holds for a volume form on $X(\CC)$,
then so does for any volume form.
\end{Claim}

\begin{proof}
We assume that the theorem
holds for a volume form $\Phi$ on $X(\CC)$.
Let $\Phi'$ be another volume form on $X(\CC)$.
Then there are constants $0 < \sigma_0 < 1$ and $\sigma_1 > 1$
with
\[
\sigma^2_0 \Phi' \leq \Phi \leq \sigma^2_1\Phi'.
\]
Thus
\[
\Vert\cdot\Vert_{L^2}^{\pmb{a} \cdot \overline{\pmb{L}} + (b-c) \overline{A}, \Phi} \leq \sigma_1\Vert\cdot\Vert_{L^2}^{\pmb{a} \cdot \overline{\pmb{L}} + (b-c) \overline{A}, \Phi'}\quad\text{and}\quad
\sigma_0 \Vert\cdot\Vert_{L^2}^{\pmb{a} \cdot \overline{\pmb{L}} -c \overline{A}, \Phi'} \leq
\Vert\cdot\Vert_{L^2}^{\pmb{a} \cdot \overline{\pmb{L}} - c \overline{A}, \Phi}.
\]
Therefore we have
\begin{multline*}
\ah\left(H^0(\pmb{a} \cdot \pmb{L} + (b-c)A), \sigma_1
\Vert\cdot\Vert^{\pmb{a} \cdot \overline{\pmb{L}} + (b-c) \overline{A},\Phi'}_{L^2}\right)  \\
\leq
\ah\left(H^0(\pmb{a} \cdot \pmb{L} + (b-c)A),
\Vert\cdot\Vert^{\pmb{a} \cdot \overline{\pmb{L}} + (b-c) \overline{A},\Phi}_{L^2}\right) 
\end{multline*}
and
\[
\ah\left(H^0(\pmb{a} \cdot \pmb{L} - cA), 
\Vert\cdot\Vert^{\pmb{a}\cdot \overline{\pmb{L}} - c \overline{A},\Phi}_{L^2}\right) \leq
\ah\left(H^0(\pmb{a}\cdot \pmb{L} - cA), \sigma_0
\Vert\cdot\Vert^{\pmb{a} \cdot \overline{\pmb{L}} - c \overline{A},\Phi'}_{L^2}\right).
\]
Note that
\[
\begin{cases}
\rank H^0(\pmb{a} \cdot \pmb{L} + (b-c)A) 
\leq \rank H^0(a_1(L_1+A) + \cdots + a_r(L_r + A))\\
\rank H^0(\pmb{a}\cdot \pmb{L} -c A) \leq \rank H^0(a_1(L_1+A)+\cdots + a_r(L_r + A)).
\end{cases}
\]
Then, by \cite[Proposition~2.1]{MoCont},
\begin{multline*}
\ah\left(H^0(\pmb{a} \cdot \pmb{L} + (b-c)A), 
\Vert\cdot\Vert^{\pmb{a} \cdot \overline{\pmb{L}} + (b-c) \overline{A},\Phi'}_{L^2}\right)  \\
\leq
\ah\left(H^0(\pmb{a}\cdot \pmb{L} + (b-c)A), \sigma_1
\Vert\cdot\Vert^{\pmb{a}\cdot \overline{\pmb{L}} + (b-c) \overline{A},\Phi'}_{L^2}\right) \\
+ \log(\sigma_1) C_1\vert \pmb{a}\vert_{1}^{d-1} + C_2 \log(\vert \pmb{a}\vert_{1}) \vert \pmb{a} \vert_{1}^{d-1}
\end{multline*}
\begin{multline*}
\ah\left(H^0(\pmb{a} \cdot \pmb{L} - cA), \sigma_0
\Vert\cdot\Vert^{\pmb{a}\cdot \overline{\pmb{L}} - c \overline{A},\Phi'}_{L^2}\right)
\leq \ah\left(H^0(\pmb{a} \cdot \pmb{L} - cA), 
\Vert\cdot\Vert^{\pmb{a}\cdot \overline{\pmb{L}} - c \overline{A},\Phi'}_{L^2}\right) \\
+\log(\sigma_0^{-1}) C_1 \vert \pmb{a}\vert_{1}^{d-1}  + C_2 \log(\vert \pmb{a} \vert_{1}) \vert \pmb{a} \vert_{1}^{d-1}.
\end{multline*}
Thus we get the claim.
\end{proof}

Therefore, we may assume that a volume form  $\Phi$ on $X(\CC)$
is given by
\[
\Phi = c_1(\overline{L}_1 + \cdots + \overline{L}_r + r\overline{A})^{d-1}.
\]
Here we fix a notation:
For a real number $\lambda$,
we set
\[
\overline{A}^{\lambda} = \overline{A} + (\OO_X, \exp(-\lambda)\vert\cdot\vert_{can}).
\]
Let us see the following claim.

\begin{Claim}
\label{claim:thm:h:0:estimate:big:2}
We may assume the following:
\begin{enumerate}
\renewcommand{\labelenumi}{\rom{(\alph{enumi})}}
\item
There is a non-zero small section $s$ of $A$ such that
$\zeros(s)$ is smooth over $\QQ$.
\rom{(}In the case where $d=1$, $\zeros(s)$ is empty on $X_{\QQ}$.\rom{)}

\item
We can find a positive integer $n$ with the following property:
For each $i=1, \ldots, r$, there is a non-zero small section
$t_i$ of $nA - L_i$ such that $t_i$ is non-zero on every
irreducible component of $\zeros(s)$.
\end{enumerate}
\end{Claim}

\begin{proof}
Since $A$ is very ample,
there is a non-zero section $s$ of
$A$ such that $\zeros(s)$ is smooth over $\QQ$.
Moreover, by \cite[Lemma~3.2]{MoCont},
we can find a positive integer $n$ with the following property:
For each $i=1, \ldots, r$,
there is a non-zero section $t_i$
of $nA - L_i$ such that
$t_i$ is non-zero on every irreducible component of $\zeros(s)$.
Let $\lambda$ be a non-negative real number
with
$\exp(-\lambda) \Vert s \Vert_{\sup} \leq 1$ and 
$\exp(-n\lambda)\Vert t_i \Vert_{\sup} \leq 1$ for $i=1,\ldots,r$.
Then $s$ and $t_i$ are small sections of
$\overline{A}^{\lambda}$ and 
$n \overline{A}^{\lambda} - \overline{L}_i$ respectively.
On the other hand,
$\overline{A} \leq \overline{A}^{\lambda}$, and
$\overline{A}^{\lambda}$ satisfy the conditions (1) and (2).
Moreover,
\[
c_1(\overline{L}_1 + \cdots + \overline{L}_r + r\overline{A}) = 
c_1(\overline{L}_1 + \cdots + \overline{L}_r + r\overline{A}^{\lambda}).
\]
Thus, by Claim~\ref{claim:thm:h:0:estimate:big:1-1},
we get our claim.
\end{proof}

\medskip
For a coherent sheaf ${\mathcal F}$ on $X$ and
a closed subscheme $Z$ of $X$,
the image of the natural homomorphism
\[
H^i(X, {\mathcal F}) \to H^i(Z, \rest{{\mathcal F}}{Z})
\]
is denoted by $I^i(Z, \rest{{\mathcal F}}{Z})$ or
$I^i(\rest{{\mathcal F}}{Z})$ for simplicity.

Let us start the proof of Theorem~\ref{thm:h:0:estimate:big:main:L:2}.
If $b = 0$, then $c=0$.
In this case, the assertion of the theorem is obvious,
so that we may assume that $b \geq 1$.
As in Claim~\ref{claim:thm:h:0:estimate:big:2}, 
let $s$ be a non-zero small section of $A$ such that
$Y := \zeros(s)$ is smooth over $\QQ$.

Here we fix constants $C_3$ and $C_4$ as follows:
\[
\begin{cases}
\rank H^0(Y, \rest{a_1(L_1+A) + \cdots + a_r(L_r + A)}{Y}) 
\leq C_3 \vert \pmb{a} \vert_{1}^{d-2}\qquad(\vert \pmb{a} \vert_{1} \geq 1) \\
\log(18) \rank H^0(Y, \rest{a_1(L_1+A) + \cdots + a_r(L_r + A)}{Y}) \\
\qquad + 2 \log\left( (\rank H^0(Y, \rest{a_1(L_1+A) + \cdots + a_r(L_r+A)}{Y})) !\right) \\
\phantom{\log(18) \rank H^0(a(L+A) + 2 \log (\rank H^0}
\leq C_4 \vert \pmb{a}\vert_{1}^{d-2} \log(\vert \pmb{a} \vert_{1}) 
\qquad(\vert \pmb{a} \vert_{1} \geq 2)
\end{cases}
\]
In the case where $d=1$,
$\rank H^0(Y, \rest{a_1(L_1+A) + \cdots + a_r(L_r+A)}{Y}) = 0$.

Let $\Vert\cdot\Vert_{L^2,\quot}^{\pmb{a}\cdot\overline{\pmb{L}}+(b-c)\overline{A}}$ be
the quotient norm of $I^0(\rest{\pmb{a}\cdot\pmb{L}+(b-c)A}{bY})$ induced by
the surjective homomorphism
$H^0(\pmb{a}\cdot\pmb{L}+(b-c)A) \to I^0(\rest{\pmb{a}\cdot\pmb{L}+(b-c)A}{bY})$ and
$L^2$-norm 
$\Vert\cdot\Vert_{L^2}^{\pmb{a}\cdot\overline{\pmb{L}}+(b-c)\overline{A}}$
of $H^0(\pmb{a}\cdot\pmb{L} + (b-c)A)$.
An exact sequence
\[
0 \to H^0(\pmb{a} \cdot \pmb{L} - cA) \overset{s^b}{\longrightarrow} H^0(\pmb{a}\cdot\pmb{L}+(b-c)A) \to I^0(\rest{\pmb{a}\cdot\pmb{L}+(b-c)A}{bY}) \to 0
\]
gives rise to an exact sequence of normed $\ZZ$-modules:
\begin{multline*}
0 \to \left(H^0(\pmb{a} \cdot \pmb{L}-cA), \Vert\cdot\Vert_{L^2, s^b, \sub}^{\pmb{a} \cdot \overline{\pmb{L}}+(b-c)\overline{A}}\right) 
\to
\left(H^0(\pmb{a}\cdot\pmb{L}+(b-c)A), \Vert\cdot\Vert_{L^2}^{\pmb{a}\cdot\overline{\pmb{L}}+(b-c)\overline{A}}\right) \\
\to \left(I^0(\rest{\pmb{a}\cdot\pmb{L}+(b-c)A}{bY}), \Vert\cdot\Vert_{L^2,\quot}^{\pmb{a}\cdot\overline{\pmb{L}}+(b-c)\overline{A}}\right) \to 0.
\end{multline*}
Since
$\rank H^0(\pmb{a}\cdot\pmb{L}-cA) \leq \rank H^0(a_1(L_1+A) + \cdots + a_r(L_r + A))$, by \cite[(4) of Proposition~2.1]{MoCont},
the above exact sequence yields
\begin{multline}
\label{thm:h:0:estimate:big:eqn:1}
\ah\left(H^0(\pmb{a} \cdot \pmb{L}+ (b-c)A), \Vert\cdot\Vert_{L^2}^{\pmb{a}\cdot\overline{\pmb{L}}+ (b-c)\overline{A}}\right) \leq
\ah\left(H^0(\pmb{a}\cdot\pmb{L} - cA), \Vert\cdot\Vert_{L^2, s^b, \sub}^{\pmb{a}\cdot\overline{\pmb{L}}+(b-c)\overline{A}}\right)   \\
+ \ah\left(I^0(\rest{\pmb{a}\cdot\pmb{L}+(b-c)A}{bY}), \Vert\cdot\Vert_{L^2,\quot}^{\pmb{a}\cdot\overline{\pmb{L}}+(b-c)\overline{A}}\right) 
+ C_2 \vert\pmb{a}\vert_{1}^{d-1}\log(\vert\pmb{a}\vert_{1})
\end{multline}
for all $\pmb{a} \in \ZZ_{\geq 0}^r$ and
$b, c \in \ZZ_{\geq 0}$ with $\vert \pmb{a} \vert_{1} \geq b \geq c \geq 0$ and
$\vert \pmb{a} \vert_{1} \geq 2$.

\medskip
Here let us consider two lemmas.

\begin{Lemma}
\label{lem:thm:h:0:estimate:big:1}
There are constants $a_0$ and $C_5$ depending only on
$\overline{\pmb{L}}$ and $\overline{A}$ such that
\begin{multline*}
\ah\left(H^0(\pmb{a}\cdot\pmb{L} - cA), \Vert\cdot\Vert_{L^2, s^b,\sub}^{\pmb{a}\cdot\overline{\pmb{L}}+(b-c)\overline{A}}\right) \leq
\ah\left(H^0(\pmb{a}\cdot\pmb{L} - cA), \Vert\cdot\Vert_{L^2}^{\pmb{a}\cdot\overline{\pmb{L}} - c \overline{A}}\right) \\
+ C_5 b \vert \pmb{a}\vert_{1}^{d-1} + C_2 \vert \pmb{a}\vert_{1}^{d-1}\log(\vert \pmb{a}\vert_{1})
\end{multline*}
for all $\pmb{a} \in \ZZ_{\geq 0}^r$ and
$b, c \in \ZZ$ with
$\vert \pmb{a} \vert_{1} \geq b \geq c \geq 0$ and
$\vert \pmb{a} \vert_{1} \geq a_0$.
\end{Lemma}

\begin{Lemma}
\label{lem:thm:h:0:estimate:big:2}
There are constants $C_6$ and $C_7$ depending only on
$\overline{\pmb{L}}$ and 
$\overline{A}$ such that
\[
\ah\left(I^0(\rest{(\pmb{a}\cdot\pmb{L} + (b-c)A)}{bY}), \Vert\cdot\Vert_{L^2,\quot}^{\pmb{a}\cdot\overline{\pmb{L}} + (b-c)\overline{A}}\right) \\
\leq C_6 b \vert \pmb{a}\vert_{1}^{d-1} + (C_4 + C_7) \vert \pmb{a}\vert_{1}^{d-1}\log(\vert \pmb{a}\vert_{1})
\]
for all $\pmb{a} \in \ZZ_{\geq 0}^r$ and
$b, c \in \ZZ$ with
$\vert \pmb{a} \vert_{1} \geq b \geq c \geq 0$ and
$\vert \pmb{a} \vert_{1} \geq 2$.
\end{Lemma}

We will prove these lemmas in the next subsection.
Assuming them, we proceed with the proof of our theorem.
Gathering \eqref{thm:h:0:estimate:big:eqn:1}, Lemma~\ref{lem:thm:h:0:estimate:big:1} and
Lemma~\ref{lem:thm:h:0:estimate:big:2},
if we put $C = C_5 + C_6$ and $D = 2C_2 + C_4 + C_7$,
then 
\begin{multline*}
\ah\left(H^0(\pmb{a}\cdot\pmb{L} + (b-c)A),
\Vert\cdot\Vert^{\pmb{a}\cdot\overline{\pmb{L}} + (b-c)\overline{A}}_{L^2}\right) \\
\leq \ah\left(H^0(\pmb{a}\cdot\pmb{L} - c A), 
\Vert\cdot\Vert^{\pmb{a}\cdot \overline{\pmb{L}} - c \overline{A}}_{L^2}\right)
+ C b \vert \pmb{a}\vert_{1}^{d-1} + D \vert \pmb{a}\vert_{1}^{d-1} \log(\vert \pmb{a}\vert_{1})
\end{multline*}
for all $\pmb{a} \in \ZZ_{\geq 0}^r$ and
$b, c \in \ZZ$ with
$\vert \pmb{a}\vert_{1}  \geq b \geq c \geq 0$ and
$\vert \pmb{a} \vert_{1} \geq a_0$.
This proves Theorem~\ref{thm:h:0:estimate:big:main:L:2}.
\end{proof}

\subsection{The proofs of Lemma~\ref{lem:thm:h:0:estimate:big:1} and
Lemma~\ref{lem:thm:h:0:estimate:big:2}}
\renewcommand{\theClaim}{\arabic{section}.\arabic{subsection}.\arabic{Claim}}
\setcounter{Theorem}{0}
In this subsection, 
we consider the proofs of
Lemma~\ref{lem:thm:h:0:estimate:big:1} and
Lemma~\ref{lem:thm:h:0:estimate:big:2}.
We keep every notation as in the previous subsection.

\begin{proof}[Proof of Lemma~\ref{lem:thm:h:0:estimate:big:1}]
Note that
$\Vert\cdot\Vert_{L^2, s^b,\sub}^{\pmb{a}\cdot\overline{\pmb{L}}+(b-c)\overline{A}} \leq
 \Vert\cdot\Vert_{L^2}^{\pmb{a}\cdot\overline{\pmb{L}} - c \overline{A}}$.
Thus, by \cite[(2) of Proposition~2.1]{MoCont},
\begin{multline*}
\ah\left(H^0(aL-cA), \Vert\cdot\Vert_{L^2}^{a\overline{L}- c\overline{A}}\right) - 
\ah\left(H^0(\pmb{a}\cdot\pmb{L}-cA), \Vert\cdot\Vert_{L^2, s^b,\sub}^{\pmb{a}\cdot\overline{\pmb{L}}+(b-c)\overline{A}}\right) \\
\hspace{20em}+ C_2 \vert \pmb{a}\vert_{1}^{d-1} \log(\vert \pmb{a}\vert_{1}) \\
\geq \achi\left(H^0(\pmb{a}\cdot\pmb{L}-cA), \Vert\cdot\Vert_{L^2}^{\pmb{a}\cdot\overline{\pmb{L}} - c\overline{A}}\right) - 
\achi\left(H^0(\pmb{a}\cdot\pmb{L}-cA), \Vert\cdot\Vert_{L^2, s^b,\sub}^{\pmb{a}\cdot\overline{\pmb{L}}+(b-c)\overline{A}}\right).
\end{multline*}
Hence, by Corollary~\ref{cor:comp:sub:L2:ball},
we obtain Lemma~\ref{lem:thm:h:0:estimate:big:1}.
\end{proof}

\medskip
\addtocounter{Theorem}{1}
\begin{proof}[Proof of Lemma~\ref{lem:thm:h:0:estimate:big:2}]
This is very complicated.
Let $k$ be an integer with
$0 \leq k < b$.
Let $\Vert\cdot\Vert_{L^2, s^k, \sub,\quot}^{\pmb{a}\cdot\overline{\pmb{L}}+(b-c)\overline{A}}$
be the quotient norm
induced by the surjective homomorphism
\[
H^0(\pmb{a}\cdot\pmb{L} + (b-c-k)A) \to I^0(Y, \rest{\pmb{a}\cdot\pmb{L} + (b-c-k)A}{Y})
\]
and the norm 
$\Vert\cdot\Vert_{L^2, s^k, \sub}^{\pmb{a}\cdot\overline{\pmb{L}}+(b-c)\overline{A}}$ of $H^0(aL + (b-c-k)A)$.
Let 
$Y'$ be the horizontal part of $Y$,
that is, the Zariski closure of $Y \cap X_{\QQ}$ in $X$.
In the case where $d=1$, $Y' = \emptyset$.
Since the kernel of
\[
I^0(Y, \rest{aL + (b-c-k)A}{Y}) \to I^0(Y', \rest{aL + (b-c-k)A}{Y'})
\]
is a torsion group,
$I^0(Y', \rest{\pmb{a}\cdot\pmb{L} + (b-c-k)A}{Y'})$ has
the same norm
$\Vert\cdot\Vert_{L^2, s^k, \sub,\quot}^{\pmb{a}\cdot\overline{\pmb{L}}+(b-c)\overline{A}}$.
Then we have the following. 

\begin{Claim}
\label{claim:thm:h:0:estimate:big:5:1}
There are constants $C'_6$ and $C'_7$ depending only on
$\overline{\pmb{L}}$ and
$\overline{A}$ such that
\[
\ah\left(I^0(Y', \rest{\pmb{a}\cdot\pmb{L} + (b-c-k)A}{Y'}), 
\Vert\cdot\Vert_{L^2, s^k, \sub,\quot}^{\pmb{a}\cdot\overline{\pmb{L}}+(b-c)\overline{A}}\right)
\leq C'_6 \vert \pmb{a}\vert_{1}^{d-1} + C'_7 \vert \pmb{a}\vert_{1}^{d-2}\log(\vert \pmb{a}\vert_{1})
\]
for all $\pmb{a} \in \ZZ_{\geq 0}^r$ and
$b, c, k\in \ZZ$ with
$\vert \pmb{a}\vert_{1} \geq b \geq c \geq 0$,
$\vert\pmb{a}\vert_{1} \geq 2$ and $0 \leq k < b$.
\end{Claim}

\begin{proof}
If $d=1$, then
$I^0(Y', \rest{aL + (b-c-k)A}{Y'} )= 0$. Thus
our assertion is obvious, so that we assume
$d \geq 2$.

Let $U$ be an open set of $X(\CC)$ such that
the closure of $U$ does not intersect with $Y'(\CC)$ and that
$U$ is not empty on each connected component of $X(\CC)$.
Applying \cite[Lemma~1.4]{MoCont} to
${L_1}_{\CC}, \ldots, {L_r}_{\CC},A_{\CC}$ and
${L_1}_{\CC}, \ldots, {L_r}_{\CC}, -A_{\CC}$,
we can find constants $D_1 \geq 1$ and $D'_1 \geq 1$
such that,
for all $\pmb{l} \in \ZZ_{\geq 0}^r$, $m \in \ZZ$ and
$u \in H^0(X(\CC), \pmb{l}\cdot\pmb{L} + mA)$,
\[
 D'_1 D_1^{\vert \pmb{l}\vert _1+ \vert m \vert} 
 \int_U \vert u \vert^2 \Phi \geq \int_{X(\CC)} \vert u \vert^2 \Phi.
\]

Since $0 < \inf\nolimits_{x \in U} \{ \vert s \vert(x)\} \leq 1$,
if we set $D_2 = 1/ \inf\nolimits_{x \in U} \{ \vert s \vert(x)\}$
then $D_2 \geq 1$. Thus, if we set $D_3 = \max\{ D_2, D_1\}$,
then, for $u \in H^0(X, \pmb{a}\cdot\pmb{L} +(b-c-k)A)$,
\begin{align*}
\int_{X(\CC)}  \vert s^k \otimes u \vert^2 \Phi & 
\geq \int_U \vert s^k \otimes u \vert ^2 \Phi \geq D_2^{-2k} \int_U \vert u \vert^2 \Phi \\
& \geq  D_2^{-2k} {D'_1}^{-1} D_1^{-(\vert \pmb{a} \vert_{1}+\vert b-c-k\vert )} \int_{X(\CC)} \vert u \vert^2 \Phi \\
& \geq {D'_1}^{-1} D_3^{-4\vert\pmb{a}\vert_{1}} 
\int_{X(\CC)} \vert u \vert^2 \Phi.
\end{align*}
The above inequality means that
\[
\Vert\cdot\Vert_{L^2, s^k,\sub}^{\pmb{a}\cdot\overline{\pmb{L}}+(b-c)\overline{A}} \geq {D'_1}^{-1/2} D_3^{-2\vert \pmb{a}\vert_{1}}
\Vert\cdot\Vert_{L^2}^{\pmb{a}\cdot\overline{\pmb{L}}+(b-c-k)\overline{A}}.
\]
Therefore, we have
\[
\Vert\cdot\Vert_{L^2, s^k,\sub,\quot}^{\pmb{a}\cdot\overline{\pmb{L}}+(b-c)\overline{A}} \geq {D'_1}^{-1/2} D_3^{-2\vert \pmb{a}\vert_{1}}
\Vert\cdot\Vert_{L^2,\quot}^{\pmb{a}\cdot\overline{\pmb{L}}+(b-c-k)\overline{A}},
\]
where 
$\Vert\cdot\Vert_{L^2,\quot}^{\pmb{a}\cdot\overline{\pmb{L}}+(b-c-k)\overline{A}}$ is the quotient norm of $I^0(Y', \rest{\pmb{a}\cdot\pmb{L} + (b - c -k)A}{Y'})$
induced by the surjective homomorphism
\[
H^0(X, \pmb{a}\cdot\pmb{L} + (b -c-k)A) \to I^0(Y', \rest{\pmb{a}\cdot\pmb{L} + (b - c -k)A}{Y'}).
\]
Note that $e^x \geq x + 1$ ($x \geq 0$).
Applying \cite[Corollary~1.1.3]{MoCont}
to ${L_1}_{\CC}, \ldots, {L_r}_{\CC},A_{\CC}$ and 
${L_1}_{\CC}, \ldots, {L_r}_{\CC}, -A_{\CC}$,
we can find constants
$D_4, D'_4  \geq 1$ such that
\begin{multline*}
 \Vert\cdot\Vert_{L^2,\quot}^{\pmb{a}\cdot\overline{\pmb{L}}+(b-c-k)\overline{A}} \geq
{D'_4}^{-1/2} D_4^{-(\vert \pmb{a}\vert _{1}+\vert b-c - k\vert)/2}\Vert\cdot\Vert_{L^2}^{\rest{\pmb{a}\cdot\overline{\pmb{L}}+(b-c-k)\overline{A}}{Y'}} \\
\geq 
{D'_4}^{-1/2} D_4^{-\vert \pmb{a}\vert_{1}}\Vert\cdot\Vert_{L^2}^{\rest{\pmb{a}\cdot\overline{\pmb{L}}+(b-c-k)\overline{A}}{Y'}}
\end{multline*}
holds on
$I^0(Y', \rest{\pmb{a}\cdot\pmb{L}+(b-c-k)A}{Y'})$.
Here a volume form of $Y$
is given by the $C^{\infty}$-hermitian invertible sheaf
$\rest{(\overline{L}_1 + \cdots + \overline{L}_r + r\overline{A})}{Y'}$.
Therefore, if we set
$D_5 = \max \{ D_3, D_4 \}$ and $D'_5 = \max \{ D'_1, D'_4 \}$,
then
\[
\Vert\cdot\Vert_{L^2, s^k,\sub,\quot}^{\pmb{a}\cdot\overline{\pmb{L}}+(b-c)\overline{A}} \geq {D'_5}^{-1}D_5^{-3\vert \pmb{a}\vert_{1}} 
\Vert\cdot\Vert_{L^2}^{\rest{\pmb{a}\cdot\overline{\pmb{L}}+(b-c-k)\overline{A}}{Y'}}
\]
holds on $I^0(Y', \rest{\pmb{a}\cdot\pmb{L} + (b-c-k)A}{Y'})$.
Note that
$\rank H^0(\rest{\pmb{a}\cdot\pmb{L} + (b-c-k)A}{Y'}) 
\leq \rank H^0(\rest{a_1(L_1+A)+ \cdots + a_r(L_r+A)}{Y'})$.
Thus, by \cite[(3) of Proposition~2.1]{MoCont},
\begin{multline*}
\ah\left(I^0(Y', \rest{\pmb{a}\cdot\pmb{L} + (b-c-k)A}{Y'}), 
\Vert\cdot\Vert_{L^2, s^k,\sub,\quot}^{\pmb{a}\cdot\overline{\pmb{L}}+(b-c)\overline{A}}\right) \\
\leq \ah\left(I^0(Y', \rest{\pmb{a}\cdot\pmb{L} + (b-c-k)A}{Y'}),
\Vert\cdot\Vert_{L^2}^{\rest{\pmb{a}\cdot\overline{\pmb{L}}+(b-c-k)\overline{A}}{Y'}}\right) \\
\qquad\qquad\qquad\qquad+ \log({D'_5}D_5^{3\vert\pmb{a}\vert_{1}} )C_3 \vert \pmb{a}\vert_{1}^{d-2}  + C_4 \vert \pmb{a}\vert_{1}^{d-2}\log(\vert \pmb{a}\vert_{1}) \\
\leq \ah\left(H^0(Y', \rest{\pmb{a}\cdot\pmb{L}+(b-c-k)A}{Y'}),
\Vert\cdot\Vert_{L^2}^{\rest{\pmb{a}\cdot\overline{\pmb{L}}+(b-c-k)\overline{A}}{Y'}}\right)\\
\qquad\qquad\qquad\qquad+\log({D'_5}D_5^{3\vert \pmb{a}\vert_{1}} )C_3 \vert \pmb{a}\vert_{1}^{d-2}  + C_4 \vert \pmb{a}\vert_{1}^{d-2}\log(\vert \pmb{a}\vert_{1}).
\end{multline*}
Let $\widetilde{Y'}$ be the normalization of $Y'$.
Let $t_1, \ldots, t_r$ be small sections as in
Claim~\ref{claim:thm:h:0:estimate:big:2}.
Then $t_i$ gives rise to an inequality
$\rest{\overline{L}_i}{\widetilde{Y'}} \leq 
\rest{n\overline{A}}{\widetilde{Y'}}$.
Therefore,
\[
\rest{\pmb{a}\cdot\overline{\pmb{L}}+(b-c-k)\overline{A}}{\widetilde{Y'}} \leq \rest{(n \vert \pmb{a}\vert_{1} + b-c-k)\overline{A}}{\widetilde{Y'}}.
\]
Thus we have
\begin{multline*}
 \ah\left(H^0(Y', \rest{\pmb{a}\cdot\pmb{L}+(b-c-k)A}{Y'}),
\Vert\cdot\Vert_{L^2}^{\rest{\pmb{a}\cdot\overline{\pmb{L}}+(b-c-k)\overline{A}}{Y'}}\right)\\
\leq
\ah\left(H^0(\widetilde{Y'}, \rest{\pmb{a}\cdot\pmb{L}+(b-c-k)A}{\widetilde{Y'}}),
\Vert\cdot\Vert_{L^2}^{\rest{\pmb{a}\cdot\overline{\pmb{L}}+(b-c-k)\overline{A}}{\widetilde{Y'}}} \right) \\
\leq 
\ah\left(H^0(\widetilde{Y'}, \rest{(n\vert \pmb{a}\vert _{1} + b-c-k)A}{\widetilde{Y'}}),
\Vert\cdot\Vert_{L^2}^{\rest{(n \vert \pmb{a}\vert_{1} + b-c-k)\overline{A}}{\widetilde{Y'}}}\right).
\end{multline*}
Moreover, by \cite[Lemma~3.3]{MoCont},
there is a constant $D_6$ such that
\[
\ah\left(H^0(\widetilde{Y'}, \rest{mA}{\widetilde{Y'}}),
\Vert\cdot\Vert_{L^2}^{m \rest{\overline{A}}{\widetilde{Y'}}}\right)
\leq D_6 m^{d-1}
\]
for all $m \geq 1$. Thus we get the claim.
\end{proof}

We would like to extend Claim~\ref{claim:thm:h:0:estimate:big:5:1} 
to $Y$. First of all, we consider the following claim.

\begin{Claim}
\label{claim:thm:h:0:estimate:big:5:2}
Let $I$ and $I'$ be the defining ideals of
$Y$ and $Y'$ respectively.
Then there is a constant $C_6''$ such that
\[
\log \# H^0(X, (\pmb{a}\cdot\pmb{L} + (b-c-k)A) \otimes I'/I) \leq C''_6 \vert \pmb{a}\vert_{1}^{d-1}
\]
for all $\pmb{a} \in \ZZ_{\geq 0}^r$ and
$b, c, k \in \ZZ$ with
$\vert \pmb{a}\vert_{1} \geq b \geq c \geq 0$,
$\vert \pmb{a} \vert_{1} \geq 2$ and $0 \leq k < b$.
\end{Claim}

\begin{proof}
Since $A$ is ample,
there is a positive integer $m_0$
such that
$H^j(X, mA \otimes I'/I) = 0$ for all $m \geq m_0$ and $j > 0$.
We set $\operatorname{Ass}_{\OO_{X}}(I'/I) = \{x_1, \ldots, x_r\}$.
Then, by \cite[Lemma~3.2]{MoCont},
there is $n_0$ independent from $a,b,c,k$ with
the following properties:
\begin{enumerate}
\renewcommand{\labelenumi}{(\roman{enumi})}
\item
For each $i=1, \ldots, r$,
there is a non-zero section $l_i$ of
$H^0(X, n_0A - L_i)$ such that
$l_i(x_j) \not= 0$ for all $j$.

\item
$2(n_0 -1) \geq m_0$.
\end{enumerate}
By (i), we have an injective homomorphism
\[
\begin{CD}
(\pmb{a}\cdot\pmb{L} + (b-c-k)A) \otimes I'/I  @>{\otimes l_1^{\otimes \pmb{a}(1)} \otimes \cdots \otimes l_r^{\otimes \pmb{a}(r)}}>> (n_0\vert \pmb{a}\vert_1 + b-c-k) A \otimes I'/I.
\end{CD}
\]
Note that
\[
n_0\vert \pmb{a}\vert_{1} + b-c-k 
\geq
(n_0 -1)\vert \pmb{a}\vert_{1} + 
\vert \pmb{a}\vert_{1} + b-c-k
\geq 2(n_0 - 1) \geq m_0.
\]
Then we have
\begin{multline*}
\log \# H^0(X, (\pmb{a}\cdot\pmb{L} + (b-c-k)A) \otimes I'/I)\\
 \leq
\log \# H^0(X, (n_0\vert \pmb{a}\vert_{1} + b-c-k) A \otimes I'/I) \\
= \sum_{j\geq 0} (-1)^j\log \# H^j(X, (n_0\vert\pmb{a}\vert_{1} + b-c-k) A \otimes I'/I).
\end{multline*}
Since $\Supp(I'/I)$ is contained in
fibers of $X \to \Spec(\ZZ)$,
by using Snapper's theorem,
we can find a polynomial $P(X) \in \RR[X]$ of degree $\leq d-1$
such that
\[
\sum_{j\geq 0} (-1)^j\log \# H^j(X, (n_0\vert \pmb{a}\vert_{1} + b-c-k) A \otimes I'/I) = P(n_0 \vert \pmb{a}\vert_{1} + b - c -k).
\]
Thus we get the claim.
\end{proof}

\begin{Claim}
\label{claim:thm:h:0:estimate:big:5}
There are constants $C_6$ and $C_7$ depending only on
$\overline{\pmb{L}}$ and
$\overline{A}$ such that
\[
\ah\left(I^0(Y, \rest{\pmb{a}\cdot\pmb{L} + (b-c-k)A}{Y}), 
\Vert\cdot\Vert_{L^2, s^k, \sub,\quot}^{\pmb{a}\cdot\overline{\pmb{L}}+(b-c)\overline{A}}\right)
\leq C_6 \vert \pmb{a}\vert_{1}^{d-1} + C_7 \vert \pmb{a}\vert_{1}^{d-2}\log(\vert \pmb{a}\vert_{1})
\]
for all $\pmb{a} \in \ZZ_{\geq 0}^r$ and
$b, c, k \in \ZZ$ with $\vert \pmb{a}\vert_{1} \geq b \geq c \geq 0$,
$\vert \pmb{a}\vert_{1} \geq 2$ and $0 \leq k < b$.
\end{Claim}

\begin{proof}
Since the kernel of the natural surjective homomorphism
\[
I^0(Y, \rest{\pmb{a}\cdot\pmb{L} + (b-c-k)A}{Y}) \to I^0(Y', \rest{\pmb{a}\cdot\pmb{L} + (b-c-k)A}{Y'})
\]
is contained in the torsion group
$H^0(X, (\pmb{a}\cdot\pmb{L} + (b-c-k)A) \otimes I'/I)$,
we have
\begin{multline*}
\ah\left(I^0(Y, \rest{\pmb{a}\cdot\pmb{L} + (b-c-k)A}{Y}), 
\Vert\cdot\Vert_{L^2, s^k, \sub,\quot}^{\pmb{a}\cdot\overline{\pmb{L}}+(b-c)\overline{A}}\right) \\
\leq
\ah\left(I^0(Y', \rest{\pmb{a}\cdot\pmb{L} + (b-c-k)A}{Y'}), 
\Vert\cdot\Vert_{L^2, s^k, \sub,\quot}^{\pmb{a}\cdot\overline{\pmb{L}}+(b-c)\overline{A}}\right) \\
+ 
\log \# H^0(X, (\pmb{a}\cdot\pmb{L} + (b-c-k)A) \otimes I'/I).
\end{multline*}
Thus our claim follows from
Claim~\ref{claim:thm:h:0:estimate:big:5:1} and
Claim~\ref{claim:thm:h:0:estimate:big:5:2}.
\end{proof}

\medskip
Let us start the proof of Lemma~\ref{lem:thm:h:0:estimate:big:2}.
A commutative diagram
\[
\begin{CD}
0 @>>> -(k+1)A @>{s^{k+1}}>> \OO_X @>>> \OO_{(k+1)Y} @>>> 0 \\
@. @VV{s}V @| @VVV @. \\
0 @>>> -kA @>{s^k}>> \OO_X @>>> \OO_{kY} @>>> 0 \\
@. @VVV @. @. @. \\
  @. \rest{-kA}{Y} @.   @.  @.
\end{CD}
\]
yields an injective homomorphism
$\alpha_k : \rest{-kA}{Y} \to \OO_{(k+1)Y}$ together with
a commutative diagram
\[
\begin{CD}
0 @>>> -kA @>{s^k}>> \OO_X @>>> \OO_{kY} @>>> 0 \\
@. @VVV @VVV @| @. \\
0 @>>> \rest{-kA}{Y} @>{\alpha_k}>> \OO_{(k+1)Y} @>>> \OO_{kY} @>>> 0,
\end{CD}
\]
where two horizontal sequence are exact.
Thus, tensoring with
$\pmb{a}\cdot\pmb{L} + (b-c)A$,
we have the following commutative diagram:
\[
{
\arraycolsep =  0.3ex
\renewcommand{\arraystretch}{1.6}
\begin{array}{ccccccccc}
0 & \to & \pmb{a}\cdot\pmb{L} + (b-c-k)A & \overset{s^k}{\longrightarrow} & \pmb{a}\cdot\pmb{L} + (b-c)A & \to & \rest{\pmb{a}\cdot\pmb{L} + (b-c)A}{kY} & \to & 0 \\
   &       & \Big\downarrow & & \Big\downarrow & & \Big|\!\Big| & & \\
0 & \to & \rest{\pmb{a}\cdot\pmb{L} + (b-c-k)A}{Y} & \overset{\alpha_k}{\longrightarrow} & \rest{\pmb{a}\cdot\pmb{L} + (b-c)A}{(k+1)Y} & \to & \rest{\pmb{a}\cdot\pmb{L} + (b-c)A}{kY} & \to & 0.
\end{array}}
\]
Therefore, we get an exact sequence
\begin{multline*}
0 \to I^0(\rest{(\pmb{a}\cdot\pmb{L} + (b-c-k)A)}{Y}) \to
I^0(\rest{(\pmb{a}\cdot\pmb{L} + (b-c)A)}{(k+1)Y}) \\
\to I^0(\rest{(\pmb{a}\cdot\pmb{L} + (b-c)A)}{kY}) \to 0.
\end{multline*}
Note that in a commutative diagram
\[
\begin{CD}
H^0(\pmb{a}\cdot\pmb{L} + (b-c-k)A) @>{s^k}>> H^0(\pmb{a}\cdot\pmb{L} + (b-c)A) \\
@VVV @VVV \\
I^0(\rest{(\pmb{a}\cdot\pmb{L} + (b-c-k)A)}{Y}) @>{\alpha_k}>>  I^0( \rest{(\pmb{a}\cdot\pmb{L} + (b-c)A)}{(k+1)Y}),
\end{CD}
\]
two vertical homomorphisms have the same kernel.
Thus, by \cite[Lemma~3.4]{MoCont},
\begin{multline*}
0 \to
\left(I^0(\rest{(\pmb{a}\cdot\pmb{L} + (b-c-k)A)}{Y}), 
\Vert\cdot\Vert_{L^2, s^k, \sub,\quot}^{\pmb{a}\cdot\overline{\pmb{L}} + (b-c)\overline{A}}\right) \\
\to \left(I^0(\rest{(\pmb{a}\cdot\pmb{L} + (b-c)A)}{(k+1)Y}), 
\Vert\cdot\Vert_{L^2,\quot}^{\pmb{a}\cdot\overline{\pmb{L}} + (b-c)\overline{A}}\right) \\
\to \left(I^0(\rest{(\pmb{a}\cdot\pmb{L} + (b-c)A)}{kY}), 
\Vert\cdot\Vert_{L^2,\quot}^{\pmb{a}\cdot\overline{\pmb{L}} + (b-c)\overline{A}}\right) \to 0
\end{multline*}
is a normed exact sequence, where, for each $i$,
the norm of
\[
\left( I^0(\rest{(\pmb{a}\cdot\pmb{L} + (b-c)A)}{iY}), \Vert\cdot\Vert_{L^2,\quot}^{\pmb{a}\cdot\overline{\pmb{L}} + (b-c)\overline{A}}\right)
\]
is the quotient norm
induced by the surjective homomorphism
$H^0(\pmb{a}\cdot\pmb{L} + (b-c)A) \to I^0(\rest{(\pmb{a}\cdot\pmb{L} + (b-c)A)}{iY})$
and
$L^2$-norm $\Vert\cdot\Vert_{L^2}^{\pmb{a}\cdot\overline{\pmb{L}} + (b-c)\overline{A}}$
of $H^0(\pmb{a}\cdot\pmb{L} + (b-c)A)$.
Note that
\[
\rank H^0(\rest{\pmb{a}\cdot\pmb{L} + (b-c-k)A}{Y}) \leq 
\rank H^0(\rest{a_1(L_1+A)+\cdots+a_r(L_r+A)}{Y}).
\]
Thus, by using \cite[(4) of Proposition~2.1]{MoCont},
\begin{multline*}
\ah\left(I^0(\rest{(\pmb{a}\cdot\pmb{L} + (b-c)A)}{(k+1)Y}), 
\Vert\cdot\Vert_{L^2,\quot}^{\pmb{a}\cdot\overline{\pmb{L}} + (b-c)\overline{A}}\right) \\
- \ah\left(I^0(\rest{(\pmb{a}\cdot\pmb{L} + (b-c)A)}{kY}), 
\Vert\cdot\Vert_{L^2,\quot}^{\pmb{a}\cdot\overline{\pmb{L}} + (b-c)\overline{A}}\right) \\
\leq \ah\left(I^0(\rest{(\pmb{a}\cdot\pmb{L} + (b-c-k)A)}{Y}), 
\Vert\cdot\Vert_{L^2, s^k, \sub,\quot}^{\pmb{a}\cdot\overline{\pmb{L}} + (b-c)\overline{A}}\right) +C_4 \vert \pmb{a}\vert_{1}^{d-2}\log(\vert \pmb{a}\vert_{1}).
\end{multline*}
Therefore, taking $\sum\nolimits_{k=1}^{b-1}$,
the above inequalities imply
\begin{multline*}
\ah\left(I^0(\rest{(\pmb{a}\cdot\pmb{L} + (b-c)A)}{bY}), 
\Vert\cdot\Vert_{L^2,\quot}^{\pmb{a}\cdot\overline{\pmb{L}} + (b-c)\overline{A}}\right) \\
-
\ah\left(I^0(\rest{(\pmb{a}\cdot\pmb{L} + (b-c)A)}{Y}), 
\Vert\cdot\Vert_{L^2,\quot}^{\pmb{a}\cdot\overline{\pmb{L}} + (b-c)\overline{A}}\right) 
\\
\leq \sum_{k=1}^{b-1}
\ah\left(I^0(\rest{(\pmb{a}\cdot\pmb{L} + (b-c-k)A)}{Y}), 
\Vert\cdot\Vert_{L^2, s^k, \sub,\quot}^{\pmb{a}\cdot\overline{\pmb{L}} + (b-c)\overline{A}}\right) \\
+ (b-1)C_4 \vert \pmb{a}\vert_{1}^{d-2}\log(\vert \pmb{a}\vert_{1}).
\end{multline*}
Hence, by using Claim~\ref{claim:thm:h:0:estimate:big:5},
we have Lemma~\ref{lem:thm:h:0:estimate:big:2}.
\end{proof}

\renewcommand{\theTheorem}{\arabic{section}.\arabic{Theorem}}
\renewcommand{\theClaim}{\arabic{section}.\arabic{Theorem}.\arabic{Claim}}
\renewcommand{\theequation}{\arabic{section}.\arabic{Theorem}}

\section{Arithmetic volume function}

Let $X$ be a $d$-dimensional projective arithmetic variety.
Let $\overline{L}$ be a continuous hermitian invertible sheaf on $X$
(cf. Conventions and terminology~\ref{CT:cont:herm:inv:sheaf}).
As mentioned in Conventions and terminology~\ref{CT:small:sec},
$\ah(\overline{L})$ is given by
\[
\ah(\overline{L}) = \log \# \left\{ s \in H^0(X, L) \mid \Vert s \Vert_{\sup} \leq 1 \right\}.
\]
In the same way as in \cite{MoCont},
the arithmetic volume $\avol(\overline{L})$ of $\overline{L}$ is defined by
\[
\avol(\overline{L}) = \limsup_{m\to\infty} \frac{\ah(m\overline{L})}{m^d/d!}.
\]
By virtue of Chen's theorem \cite{Chen}, $\avol(\overline{L})$ is actually given by
\[
\avol(\overline{L}) = \lim_{m\to\infty} \frac{\ah(m\overline{L})}{m^d/d!}.
\]

Fix a volume form $\Phi$ on $X(\CC)$.
For a real number $p \geq 1$, let $\Vert\cdot\Vert_{L^p}$ be the $L^p$-norm of
$H^0(X(\CC), L_{\CC})$ induced by the hermitian metric $\vert\cdot\vert$ of $\overline{L}$
and
$\Phi$, that is,
\[
\Vert s \Vert_{L^p} := \left( \int_{X(\CC)} \vert s \vert^p \Phi \right)^{1/p}
\qquad (s \in H^0(X(\CC), L_{\CC})).
\]
Similarly as above, we can define a volume with respect to the $L^p$-norm to be
\[
\avol_{L^p}(\overline{L}) = \limsup_{m\to\infty}
\frac{\log \#\{ s \in H^0(X, mL) \mid \Vert s \Vert_{L^p} \leq 1\}}{m^d/d!}.
\]
Then, using Gromov's inequality \cite[Corollary~1.1.2]{MoCont} and
Chen's result \cite{Chen}, it is easy to see that
$\avol_{L^p}(\overline{L}) = \avol(\overline{L})$ and
\[
\avol_{L^p}(\overline{L}) = \lim_{m\to\infty}
\frac{\log \#\{ s \in H^0(X, mL) \mid \Vert s \Vert_{L^p} \leq 1\}}{m^d/d!}.
\]

For the definitions of $C^0(X)$ and $\aPic_{C^0}(X)$, see
Conventions and terminology~\ref{CT:cont:herm:inv:sheaf}.
Let $\overline{\OO} : C^0(X) \to \aPic_{C^0}(X)$ be a homomorphism
defined by
\[
f \mapsto \overline{\OO}(f) := (\OO_X, \exp(-f)\vert\cdot\vert_{can}).
\]
Let $\zeta : \aPic_{C^0}(X) \to \Pic(X)$ be a natural homomorphism
given by forgetting the equipped hermitian metric.
Then the following sequence
\addtocounter{Theorem}{1}
\begin{equation}
\label{eqn:exact:C0:aPic:Pic}
\begin{CD}
C^0(X) @>{\overline{\OO}}>> \aPic_{C^0}(X) @>{\zeta}>> \Pic(X) @>>> 0
\end{CD}
\end{equation}
is exact.
It is easy to see that
\addtocounter{Theorem}{1}
\begin{equation}
\label{eqn:kernel:OO}
\Ker(\overline{\OO}) = \left\{ \log \vert u \vert \mid u \in H^0(X, \OO_{X}^{\times}) \right\}.
\end{equation}
Note that the natural homomorphism $C^0(X) \otimes_{\ZZ} \QQ \to C^0(X)$ ($f \otimes a \mapsto af$)
is an isomorphism.
Thus \eqref{eqn:exact:C0:aPic:Pic} yields an exact sequence
\addtocounter{Theorem}{1}
\begin{equation}
\label{eqn:exact:C0:aPic:Pic:QQ}
\begin{CD}
C^0(X) @>{\overline{\OO}}>> \aPic_{C^0}(X) \otimes_{\ZZ} \QQ @>{\zeta \otimes \operatorname{id}_{\QQ}}>> \Pic(X) \otimes_{\ZZ} \QQ @>>> 0
\end{CD}
\end{equation}
together with the following commutative diagram:
\[
\begin{CD}
C^0(X) @>{\overline{\OO}}>> \aPic_{C^0}(X) @>{\zeta}>> \Pic(X) @>>> 0 \\
@| @VVV @VVV \\
C^0(X) @>{\overline{\OO}}>> \aPic_{C^0}(X) \otimes_{\ZZ} \QQ @>{\zeta \otimes \operatorname{id}_{\QQ}}>> \Pic(X) \otimes_{\ZZ} \QQ @>>> 0. 
\end{CD}
\]
Let us begin with the following lemma.

\begin{Lemma}
\label{lem:Stone:Weierstrass}
Let $\overline{L}$ be a continuous hermitian $\QQ$-invertible sheaf on $X$
\rom{(}cf. Conventions and terminology~\rom{\ref{CT:cont:herm:inv:sheaf}}\rom{)}.
For any positive real number $\epsilon$,
there is $\phi \in C^0(X)$ such that $\Vert \phi \Vert_{\sup} \leq\epsilon$ and
$\overline{L} + \overline{\OO}(\phi)$ is $C^{\infty}$.
\end{Lemma}

\begin{proof}
Let us choose a positive integer $n$ and
a continuous hermitian invertible sheaf $\overline{M}$ such that
$n \overline{L} = \overline{M} \otimes 1$ in $\aPic_{C^0}(X) \otimes_{\ZZ} \QQ$.
Let $\vert\cdot\vert$ be the metric of $\overline{M}$ and fix a
$C^{\infty}$-metric $\vert\cdot\vert_0$ of $M$.
Then there is $f \in C^0(X)$ such that $\vert\cdot\vert = \exp(-f)\vert\cdot\vert_0$.
By Stone-Weierstrass theorem,
there is $g \in C^{\infty}(X)$ such that $\Vert f - g \Vert_{\sup} \leq n \epsilon$.
Note that $\exp(-(f-g)) \vert\cdot\vert = \exp(g)\vert\cdot\vert_0$.
Thus $\overline{M} + \overline{\OO}(f-g)$ is a $C^{\infty}$-hermitian invertible sheaf.
Therefore, if we set $\phi = (f-g)/n$, then $\Vert \phi \Vert_{\sup} \leq \epsilon$ and
\[
n(\overline{L} + \overline{\OO}(\phi)) = \overline{M} \otimes 1 + \overline{\OO}(f-g) =
(\overline{M} + \overline{\OO}(f-g)) \otimes 1.
\]
Hence the lemma follows.
\end{proof}

\begin{Proposition}
\label{prop:avol:hom:cont}
For a continuous hermitian invertible sheaf $\overline{L}$ on $X$,
we have the following:
\begin{enumerate}
\renewcommand{\labelenumi}{\rom{(\arabic{enumi})}}
\item
$\vert \avol(\overline{L} + \overline{\OO}(f)) - \avol(\overline{L}) \vert \leq d \Vert f \Vert_{\sup} \vol(L_{\QQ})$
for $f \in C^0(X)$.

\item
$\avol(n\overline{L}) = n^d \avol(\overline{L})$ for a non-negative integer $n$.

\item
If $\nu : X' \to X$ is  a birational morphism of projective arithmetic varieties, then
$\avol(\nu^*(\overline{L})) = \avol(\overline{L})$  for any $\overline{L} \in \aPic_{C^0}(X)$.
\end{enumerate}
\end{Proposition}

\begin{proof}
(1) We set $\lambda = \Vert f \Vert_{\sup}$. Then $-\lambda \leq f \leq \lambda$.
Thus it is easy to see that
\[
\avol(\overline{L} + \overline{\OO}(-\lambda)) \leq \avol(\overline{L} + \overline{\OO}(f))
\leq  \avol(\overline{L} + \overline{\OO}(\lambda)).
\]
Therefore, it is a consequence of
\cite[Proposition~2.1, (3)]{MoCont} in the same way as in 
\cite[Proposition~4.2]{MoCont}.

(2) 
Fix a positive integer $n$.
By Lemma~\ref{lem:Stone:Weierstrass}, for any positive number $\epsilon$,
there is $\phi \in C^0(X)$ such that $\Vert \phi \Vert_{\sup} \leq \epsilon$ and
$\overline{L} + \overline{\OO}(\phi)$ is $C^{\infty}$.
Hence, by \cite[Propostion~4.8]{MoCont},
$\avol(n(\overline{L} + \overline{\OO}(\phi))) = n^d \avol(\overline{L} + \overline{\OO}(\phi))$.
Therefore, using (1), we have
\begin{multline*}
\vert \avol(n\overline{L}) - n^d\avol(\overline{L}) \vert  \leq \vert \avol(n\overline{L}) - \avol(n(\overline{L} + \overline{\OO}(\phi))) \vert + \vert \avol(n(\overline{L} + \overline{\OO}(\phi))) - n^d\avol(\overline{L}) \vert \\
\hspace{8em} \leq d \Vert n\phi \Vert_{\sup} \vol(nL_{\QQ}) + n^d\vert \avol(\overline{L} + \overline{\OO}(\phi)) - \avol(\overline{L}) \vert \\
\hspace{14em} \leq d (n \Vert \phi \Vert_{\sup})(n^{d-1} \vol(L_{\QQ})) + n^d (d \Vert \phi \Vert_{\sup} \vol(L_{\QQ})) \\
= 2n^dd\vol(L_{\QQ}) \Vert \phi \Vert_{\sup} \leq (2n^dd\vol(L_{\QQ}))\epsilon.
\end{multline*}
Here $\epsilon$ is arbitrary. The above estimation implies (2).

(3) By \cite[Theorem~4.3]{MoCont}, (3) holds if $\overline{L}$ is $C^{\infty}$.
For a positive real number $\epsilon$, by Lemma~\ref{lem:Stone:Weierstrass}, we can find
$\phi \in C^0(X)$  such that
$\overline{L} + \overline{\OO}(\phi)$ is $C^{\infty}$ and $\Vert \phi \Vert_{\sup} \leq \epsilon$.
Then, by (1),
\[
\begin{cases}
\vert \avol(\overline{L} + \overline{\OO}(\phi)) - \avol(\overline{L}) \vert \leq d\epsilon \vol(L_{\QQ}), \\
\vert \avol(\nu^*(\overline{L} + \overline{\OO}(\phi))) - \avol(\nu^*(\overline{L})) \vert \leq d \Vert \nu^*(\phi)\Vert_{\sup} \vol(\nu^*(L_{\QQ})) = d \epsilon \vol(L_{\QQ}).
\end{cases}
\]
Thus,
\begin{multline*}
\vert \avol(\nu^*(\overline{L})) - \avol(\overline{L}) \vert
\leq
\vert  \avol(\nu^*(\overline{L})) - \avol(\nu^*(\overline{L} + \overline{\OO}(\phi))) \vert \\
+
\vert \avol(\nu^*(\overline{L} + \overline{\OO}(\phi))) - \avol(\overline{L} + \overline{\OO}(\phi)) \vert 
+
\vert \avol(\overline{L} + \overline{\OO}(\phi)) - \avol(\overline{L}) \vert \\
\leq \epsilon (2d \vol(L_{\QQ})). 
\end{multline*}
Therefore we get (3).
\end{proof}

By virtue of (2) of the above proposition, 
if $\overline{L} \otimes \alpha = \overline{M} \otimes \beta$ in $\aPic_{C^0}(X) \otimes_{\ZZ} \QQ$
for $\overline{L}, \overline{M} \in \aPic_{C^0}(X)$ and $\alpha, \beta \in \QQ_{\geq 0}$,
then $\alpha^d \avol(\overline{L}) = \beta^d \avol(\overline{M})$.
Indeed, we choose a positive integer $m$ such that $m \alpha, m \beta \in \ZZ$.
Then 
\[
m \alpha \overline{L} \otimes 1 = m(\overline{L} \otimes \alpha) = m(\overline{M} \otimes \beta)
= m\beta \overline{M} \otimes 1.
\]
Thus there is a positive integer $n$ with $nm\alpha \overline{L} = nm\beta \overline{M}$,
which implies $(nm\alpha)^d \avol(\overline{L}) = (nm\beta)^d \avol(\overline{M})$
by (2) of the above proposition. Hence $\alpha^d \avol(\overline{L}) = \beta^d \avol(\overline{M})$.
Therefore, we can define $\avol : \aPic_{C^0}(X) \otimes_{\ZZ} \QQ \to \RR$ to be
$\avol(\overline{L} \otimes \alpha) = \alpha^d \avol(\overline{L})$, where
$\overline{L} \in \aPic_{C^0}(X)$ and $\alpha \in \QQ_{\geq 0}$.
By the definition of $\avol$,  a diagram
\[
\xymatrix{
\aPic_{C^0}(X) \ar[r]^(.55){\avol} \ar[d] & \RR \\
\aPic_{C^0}(X) \otimes_{\ZZ} \QQ \ar[ru]_{\avol} & \\
}
\]
is commutative.
The following proposition is an immediate consequence of
the above proposition.

\begin{Proposition}
\label{prop:avol:hom:cont:QQ}
For $\overline{L} \in \aPic_{C^0}(X) \otimes_{\ZZ} \QQ$, the following hold:
\begin{enumerate}
\renewcommand{\labelenumi}{\rom{(\arabic{enumi})}}
\item
$\vert \avol(\overline{L} + \overline{\OO}(f)) - \avol(\overline{L}) \vert \leq d \Vert f \Vert_{\sup} \vol(L_{\QQ})$
for $f  \in C^0(X)$.

\item
$\avol(a\overline{L}) = a^d \avol(\overline{L})$ for a non-negative rational number $a$.

\item
If $\nu : X' \to X$ is a birational morphism of projective arithmetic varieties, then
$\avol(\nu^*(\overline{L})) = \avol(\overline{L})$  for any $\overline{L} \in \aPic_{C^0}(X) \otimes_{\ZZ} \QQ$.
\end{enumerate}
\end{Proposition}

Let $\overline{L}$ be a continuous hermitian invertible sheaf on $X$.
We say $\overline{L}$ is {\em effective} if there is a non-zero section
$s \in H^0(X, L)$ with $\Vert s \Vert_{\sup} \leq 1$.
For $\overline{L}_1, \overline{L}_2 \in \aPic_{C^0}(X)$,
if $\overline{L}_1 - \overline{L}_2$ is effective,
then it is denoted by $\overline{L}_1 \geq \overline{L}_2$ or
$\overline{L}_2 \leq \overline{L}_1$.
Moreover, for $\overline{M} \in \aPic_{C^0}(X) \otimes_{\ZZ} \QQ$,
we say $\overline{M}$ is {\em $\QQ$-effective} 
if there are a positive integer $n$ and $\overline{L} \in \aPic_{C^0}(X)$ such that
$\overline{L}$ is effective and
$n\overline{M}  = \overline{L} \otimes 1$ in $\aPic_{C^0}(X) \otimes_{\ZZ} \QQ$.
For $\overline{M}_1, \overline{M}_2 \in \aPic_{C^0}(X) \otimes_{\ZZ} \QQ$,
if $\overline{M}_1 - \overline{M}_2$ is $\QQ$-effective,
then it is denoted by
$\overline{M}_1 \geq_{\QQ} \overline{M}_2$ or
$\overline{M}_2 \leq_{\QQ} \overline{M}_1$.

\begin{Proposition}
\label{prop:ineq:Pic:C:0}
\begin{enumerate}
\renewcommand{\labelenumi}{\rom{(\arabic{enumi})}}
\item
If $\overline{L}_1 \geq \overline{L}_2$
for $\overline{L}_1, \overline{L}_2 \in \aPic_{C^0}(X)$, 
then $\ah(\overline{L}_1) \geq \ah(\overline{L}_2)$ and
$\avol(\overline{L}_1) \geq \avol(\overline{L}_2)$.

\item
If  $\overline{M}_1 \geq_{\QQ} \overline{M}_2$
for  $\overline{M}_1, \overline{M}_2 \in \aPic_{C^0}(X) \otimes_{\ZZ} \QQ$, 
then $\avol(\overline{M}_1) \geq \avol(\overline{M}_2)$.
\end{enumerate}
\end{Proposition}

\begin{proof}
(1) is easily checked.
Let us consider (2). Since $\overline{M}_1 \geq_{\QQ} \overline{M}_2$,
there are a positive integer $n$ and $\overline{L} \in \aPic_{C^0}(X)$ such that
$\overline{L}$ is effective and
$n(\overline{M}_1 - \overline{M}_2)  = \overline{L} \otimes 1$.
Moreover, we can find a positive integer $m$ and
$\overline{L}_1, \overline{L}_2 \in \aPic_{C^0}(X)$ such that
$m \overline{M}_1 = \overline{L}_1 \otimes 1$ and
$m \overline{M}_2 = \overline{L}_2 \otimes 1$.
Then
\[
m \overline{L} \otimes 1 = mn(\overline{M}_1 - \overline{M}_2) = n(\overline{L}_1 - \overline{L}_2) \otimes 1.
\]
Thus there is a positive integer $l$ with
$lm\overline{L} = l n(\overline{L}_1 - \overline{L}_2)$, which implies that
$\avol(ln \overline{L}_1) \geq \avol(ln \overline{L}_2)$. Therefore,
\begin{multline*}
\avol(lnm\overline{M}_1) = \avol(ln \overline{L}_1 \otimes 1) = \avol(ln \overline{L}_1) \\
\geq \avol(ln \overline{L}_2) = \avol(ln \overline{L}_2 \otimes 1) =
\avol(lnm\overline{M}_2).
\end{multline*}
Hence, using the homogeneity of $\avol$, we have $\avol(\overline{M}_1) \geq \avol(\overline{M}_2)$.
\end{proof}

\renewcommand{\theequation}{\arabic{section}.\arabic{Theorem}.\arabic{Claim}}

\section{Uniform continuity of the arithmetic volume function}

Let $X$ be a $d$-dimensional projective arithmetic variety.
The purpose of this section is to prove the uniform continuity of
the arithmetic volume function $\avol : \aPic_{C^0}(X) \otimes_{\ZZ} \QQ \to \RR$
in the following sense;
$\avol$ is uniformly continuous on any bounded set in any finite dimensional
vector subspace of $\aPic_{C^0}(X) \otimes_{\ZZ} \QQ$ (cf. Theorem~\ref{thm:cont:C:0}).

Let us begin with the following lemma.

\begin{Lemma}
\label{lem:vol:bound:alg}
Let $V$ be a projective variety over a field and
let $\pmb{L} = (L_1, \ldots, L_r)$ be a finite sequence of $\QQ$-invertible sheaves on $V$.
Then there is a constant $C$ depending only on
$V$ and $\pmb{L}$ such that
$\vol(\pmb{a} \cdot \pmb{L}) \leq C \vert \pmb{a} \vert_{1}^{\dim V}$
for all $\pmb{a}  \in \RR^r$.
\end{Lemma}

\begin{proof}
We set $f(\pmb{a}) = \vol(\pmb{a} \cdot \pmb{L})$ for $\pmb{a} \in \RR^r$.
It is well known that $f$ is  a continuous and homogeneous
function of degree $\dim V$ on $\RR^{r}$
(cf. \cite{Laz}).
We set $K = \{ \pmb{x} \in \RR^r \mid \vert \pmb{x} \vert_{1} = 1 \}$.
Since $K$ is compact, if we set $C = \sup_{\pmb{x} \in K}
f(\pmb{x})$, then, for $\pmb{y} \in \RR^r \setminus \{ 0 \}$,
$f(\pmb{y}/\vert \pmb{y} \vert_{1}) \leq C$.
Thus, since $f$ is homogeneous of degree $\dim V$,
we have $f(\pmb{y}) \leq C \vert \pmb{y} \vert_{1}^{\dim V}$.
\end{proof}

Next we consider the 
strong estimate of $\avol$ in $\aPic(X) \otimes_{\ZZ} \QQ$.

\begin{Theorem}
\label{thm:C:infty:cont}
Let $\overline{\pmb{L}} = 
(\overline{L}_1, \ldots, \overline{L}_r)$ and
$\overline{\pmb{A}} =
(\overline{A}_1, \ldots, \overline{A}_{r'})$
be finite sequences of $C^{\infty}$-hermitian
$\QQ$-invertible sheaves on $X$.
Then there is a positive constant $C$
depending only on $X$,
$\overline{\pmb{L}}$ and
$\overline{\pmb{A}}$
such that
\[
\vert \avol(\pmb{a} \cdot \overline{\pmb{L}} +
\pmb{\delta} \cdot \overline{\pmb{A}})
- \avol(\pmb{a} \cdot \overline{\pmb{L}}) \vert
\leq C \vert (\pmb{a}, \pmb{\delta}) \vert_{1}^{d-1}\vert \pmb{\delta} \vert_{1}
\]
for all $\pmb{a} \in \QQ^{r}$ and
$\pmb{\delta} \in \QQ^{r'}$.
\end{Theorem}

\begin{proof}
First let us see the following claim:

\begin{Claim}
Let $\overline{A}$ be a $\QQ$-effective $C^{\infty}$-hermitian $\QQ$-invertible sheaf on $X$.
Then there is a positive constant $C_1$
depending only on $X$,
$\overline{\pmb{L}}$ and
$\overline{A}$
such that
\[
\vert \avol(\pmb{a} \cdot \overline{\pmb{L}} +
\delta \cdot \overline{A})
- \avol(\pmb{a} \cdot \overline{\pmb{L}}) \vert
\leq C_1 \vert (\pmb{a}, \delta) \vert_{1}^{d-1} \vert \delta \vert
\]
for all $\pmb{a} \in \QQ^{r}$ and
$\delta \in \QQ$. 
\end{Claim}

\begin{proof}
Let $\nu : X' \to X$ be a generic resolution of singularities of $X$.
Then, by \cite[Theorem~4.3]{MoCont} or Proposition~\ref{prop:avol:hom:cont:QQ},
\[
\avol(\pmb{a} \cdot \overline{\pmb{L}}) =
\avol(\pmb{a} \cdot \nu^*(\overline{\pmb{L}})),\quad
\avol(\pmb{a} \cdot \overline{\pmb{L}} +
\delta \cdot \overline{A}) =
\avol(\pmb{a} \cdot \nu^*(\overline{\pmb{L}}) +
\delta \cdot \nu^*(\overline{A})).
\]
Thus we may assume that $X$ is generically smooth.
Moreover, since
\[
\vert \avol(\pmb{a} \cdot n \overline{\pmb{L}} +
\delta n \overline{A})
- \avol(\pmb{a} \cdot n \overline{\pmb{L}}) \vert
=
n^d\vert \avol(\pmb{a} \cdot \overline{\pmb{L}} +
\delta\overline{A})
- \avol(\pmb{a} \cdot \overline{\pmb{L}}) \vert
\]
for a positive integer $n$,
we may assume that
$\overline{L}_1, \ldots, \overline{L}_r,
\overline{A}
\in \aPic(X)$ and $\overline{A}$ is effective.

We set $\overline{\pmb{L}}' = (\overline{L}_1, \ldots, \overline{L}_r, 0)$.
By Theorem~\ref{thm:h:0:estimate:big:main},
there are a positive constants
$a'_0$, $C_1'$ and $D'_1$ depending only on
$X$, $\overline{L}_1, \ldots, \overline{L}_r$ and
$\overline{A}$ such that
\[
\ah\left(\pmb{a}' \cdot \overline{\pmb{L}}' + (b-c) \overline{A} \right) \leq
\ah\left(\pmb{a}' \cdot \overline{\pmb{L}}' - c \overline{A}\right)
+ C'_1 b \vert \pmb{a}' \vert_{1}^{d-1} + D'_1 \vert \pmb{a}' \vert_{1}^{d-1} \log(\vert \pmb{a}' \vert_{1})
\]
for all $\pmb{a}' \in \ZZ^{r+1}$ and $b, c \in \ZZ$
with $\vert \pmb{a}' \vert_{1} \geq b \geq c \geq 0$ and
$\vert \pmb{a}' \vert_{1} \geq a_0$.

If $\delta = 0$, then the assertion of the claim is obvious,
so that we assume that $\delta \not= 0$.
Let $m_0$ be a positive integer such that
$m_0/\vert\delta\vert \in \ZZ$ and $(m_0/\vert\delta\vert) \pmb{a} \in \ZZ^r$.

Applying the above estimate to
the case where $\pmb{a}' = m(m_0/\vert\delta\vert)(\pmb{a},\delta)$, $b = mm_0$ and
$c=0$ ($m \gg 0$),
we have
\[
\avol((m_0/\vert\delta\vert) \pmb{a} \cdot \overline{\pmb{L}} + m_0 \overline{A})
\leq \avol((m_0/\vert\delta\vert) \pmb{a} \cdot\overline{\pmb{L}}) + d! C'_1m_0^d
(1/\vert \delta \vert)^{d-1}\vert (\pmb{a},\delta) \vert_{1}^{d-1}
\]
becuase $ m(m_0/\vert\delta\vert)(\pmb{a},\delta) \cdot \overline{\pmb{L}}' =  m(m_0/\vert\delta\vert)\pmb{a} \cdot \overline{\pmb{L}}$. Thus, using the homogeneity of $\avol$, we obtain
\[
0 \leq 
\avol(\pmb{a} \cdot \overline{\pmb{L}} + \vert\delta\vert \overline{A})
- \avol(\pmb{a} \cdot\overline{\pmb{L}}) \leq d! C'_1
\vert \delta \vert \vert (\pmb{a},\delta) \vert_{1}^{d-1}.
\]
Next, applying the above estimate to
the case where $\pmb{a}' = m(m_0/\vert\delta\vert)(\pmb{a},\delta)$, $b = mm_0$ and
$c=mm_0$ ($m \gg 0$),
we have
\[
0 \leq \avol(\pmb{a} \cdot \overline{\pmb{L}}) - \avol(\pmb{a} \cdot\overline{\pmb{L}}-\vert\delta\vert \overline{A})
\leq  d! C'_1
\vert\delta\vert \vert (\pmb{a},\delta) \vert_{1}^{d-1}.
\]
Thus the claim follows.
\end{proof}

Next we consider a general case.
We can find $C^{\infty}$-hermitian $\QQ$-invertible sheaves
$\overline{A}'_1, \overline{A}''_1,
\ldots, \overline{A}'_{r'}, \overline{A}''_{r'}$ such that
$\overline{A}_i = \overline{A}'_i - \overline{A}''_i$,
$\overline{A}'_i\geq_{\QQ}0$ and $\overline{A}''_i \geq_{\QQ} 0$ for all $i=1,\ldots,r'$.
Then, since $\vert (\pmb{a}, \pmb{\delta}, -\pmb{\delta}) \vert_1 \leq 2 \vert (\pmb{a}, \pmb{\delta}) \vert_1$,
$\vert(\pmb{\delta}, -\pmb{\delta})\vert_{1} = 2\vert \pmb{\delta} \vert_{1}$ and
\[
\pmb{a} \cdot \overline{\pmb{L}} + \pmb{\delta} \cdot \overline{\pmb{A}}
= \pmb{a} \cdot \overline{\pmb{L}} + \pmb{\delta} \cdot \overline{\pmb{A}}'
+  (-\pmb{\delta}) \cdot \overline{\pmb{A}}'',
\]
we may assume that $\overline{A}_i \geq_{\QQ} 0$ for all $i$.
We set $\overline{B} = \overline{A}_1 + \cdots + \overline{A}_{r'}$.
Then we have $-\vert \pmb{\delta} \vert_{1} \overline{B}
\leq \pmb{\delta} \cdot \overline{\pmb{A}} \leq \vert \pmb{\delta} \vert_{1}
\overline{B}$, which implies that
\[
\avol(\pmb{a} \cdot \overline{\pmb{L}} -\vert \pmb{\delta} \vert_{1} \overline{B} )\leq
\avol(\pmb{a} \cdot \overline{\pmb{L}} + \pmb{\delta} \cdot \overline{\pmb{A}})
\leq
\avol(\pmb{a} \cdot \overline{\pmb{L}} + \vert \pmb{\delta} \vert_{1} \overline{B}).
\] 
Thus the theorem follows from the previous claim.
\end{proof}

\begin{Corollary}
\label{cor:C:infty:cont}
Let $\overline{\pmb{L}} = 
(\overline{L}_1, \ldots, \overline{L}_r)$ and
$\overline{\pmb{A}} =
(\overline{A}_1, \ldots, \overline{A}_{r'})$
be finite sequences of $C^{\infty}$-hermitian
$\QQ$-invertible sheaves on $X$.
Then there are positive constants $C$ and $C'$
depending only on $X$,
$\overline{\pmb{L}}$ and
$\overline{\pmb{A}}$
such that
\[
\vert \avol(\pmb{a} \cdot \overline{\pmb{L}} +
\pmb{\delta} \cdot \overline{\pmb{A}} + \overline{\OO}(g))
- \avol(\pmb{a} \cdot \overline{\pmb{L}}) \vert
\leq C \vert (\pmb{a}, \pmb{\delta}) \vert_{1}^{d-1}\vert \pmb{\delta} \vert_{1} + 
C'  \vert (\pmb{a}, \pmb{\delta}) \vert_{1}^{d-1}
\Vert g \Vert_{\sup}
\]
for all $\pmb{a} \in \QQ^{r}$,
$\pmb{\delta} \in \QQ^{r'}$ and $g \in C^0(X)$.
\end{Corollary}

\begin{proof}
By (1) of Proposition~\ref{prop:avol:hom:cont:QQ} and Lemma~\ref{lem:vol:bound:alg}, 
there is a positive constant $C'$ 
depending only on $d$, $\pmb{L}_{\QQ}$ and $\pmb{A}_{\QQ}$ such that
\[
\vert \avol(\pmb{a} \cdot \overline{\pmb{L}} +
\pmb{\delta} \cdot \overline{\pmb{A}} + \overline{\OO}(g))
- \avol(\pmb{a} \cdot \overline{\pmb{L}} + \pmb{\delta} \cdot \overline{\pmb{A}} ) \vert
\leq C' \Vert g \Vert_{\sup} \vert (\pmb{a}, \pmb{\delta} ) \vert^{d-1}_1 
\]
for all $\pmb{a} \in \QQ^{r}$,
$\pmb{\delta} \in \QQ^{r'}$ and $g \in C^0(X)$.
Therefore, the corollary follows from Theorem~\ref{thm:C:infty:cont}.
\end{proof}

\begin{Theorem}
\label{thm:cont:C:0}
Let $\overline{\pmb{L}} = (\overline{L}_1, \ldots, \overline{L}_r)$
and $\overline{\pmb{A}} = (\overline{A}_1, \ldots, \overline{A}_{r'})$
be finite sequences of continuous hermitian $\QQ$-invertible sheaves on $X$.
Let $B$ be a bounded set in $\QQ^r$. Then,
for any positive real number $\epsilon$,
there are positive real numbers $\delta$ and $\delta'$ such that
\[
\left\vert \avol(\pmb{a} \cdot \overline{\pmb{L}} + \pmb{\delta} \cdot
\overline{\pmb{A}} + \overline{\OO}(g)) - \avol(\pmb{a} \cdot \overline{\pmb{L}})\right\vert
\leq \epsilon
\]
for all $\pmb{a} \in B$, $\pmb{\delta} \in \QQ^{r'}$ and 
$g \in C^0(X)$ with $\vert \pmb{\delta}\vert_1 \leq \delta$ and $\Vert g \Vert_{\sup} \leq \delta'$.
In particular, if we set $f(\pmb{x}) = \avol(\pmb{x} \cdot \overline{\pmb{L}})$ for
$ \pmb{x} \in \QQ^r$, then
$f$ is uniformly continuous on $B$.
\end{Theorem}

\begin{proof}
By Lemma~\ref{lem:vol:bound:alg}, there is a constant $C_1$ such that
\[
\begin{cases}
\vol((\pmb{a}\cdot \pmb{L})_{\QQ}) \leq C_1 \vert \pmb{a} \vert_{1}^{d-1}, \\
\vol((\pmb{a} \cdot \pmb{L} + \pmb{\delta}\cdot \pmb{A})_{\QQ})
\leq C_1 \vert (\pmb{a}, \pmb{\delta}) \vert_{1}^{d-1}
\end{cases}
\]
for all $\pmb{a} \in \QQ^r$ and $\pmb{\delta} \in \QQ^{r'}$.
We set $M = \sup \{ \vert \pmb{a} \vert_1 \mid \pmb{a} \in B \}$.
By Lemma~\ref{lem:Stone:Weierstrass},
we can find $\phi_1, \ldots, \phi_r, \psi_1, \ldots, \psi_{r'}  \in C^0(X)$ such that
\[
\overline{\pmb{L}}^{\pmb{\phi}} =
(\overline{L}_1+ \overline{\OO}(\phi_1), \ldots, \overline{L}_r+ \overline{\OO}(\phi_r))
\quad\text{and}\quad
\overline{\pmb{A}}^{\pmb{\psi}} = (\overline{A}_1+ \overline{\OO}(\psi_1), \ldots,\overline{A}_{r'}+ \overline{\OO}(\psi_{r'}))
\]
are $C^{\infty}$ and that
\[
\Vert \phi_i \Vert_{\sup} \leq \frac{\epsilon}{3C_1d(M+1)^d}
\quad\text{and}\quad
\Vert \psi_j \Vert_{\sup} \leq \frac{\epsilon}{3C_1d(M + 1)^d}
\]
for all $i$ and $j$.
Since
\[
\pmb{a} \cdot \overline{\pmb{L}}^{\pmb{\phi}} =
\pmb{a}  \cdot \overline{\pmb{L}} + \overline{\OO}(\pmb{a} \cdot \pmb{\phi})
\quad\text{and}\quad
\pmb{a} \cdot \overline{\pmb{L}}^{\pmb{\phi}} + 
\pmb{\delta} \cdot \overline{\pmb{A}}^{\pmb{\psi}} =
\pmb{a} \cdot \overline{\pmb{L}} + \pmb{\delta} \cdot
\overline{\pmb{A}} + \overline{\OO}(\pmb{a} \cdot \pmb{\phi} + \pmb{\delta} \cdot \pmb{\psi}),
\]
by (1) of Proposition~\ref{prop:avol:hom:cont:QQ}, we have
\[
\begin{cases}
\vert \avol(\pmb{a} \cdot \overline{\pmb{L}}^{\pmb{\phi}}) -
\avol(\pmb{a} \cdot \overline{\pmb{L}}) \vert \leq d C_1 \Vert \pmb{a} \cdot \pmb{\phi}  \Vert_{\sup} \vert \pmb{a} \vert^{d-1}_1 \\
\vert \avol(\pmb{a} \cdot \overline{\pmb{L}}^{\pmb{\phi}} + \pmb{\delta} \cdot \overline{\pmb{A}}^{\pmb{\psi}} + \overline{\OO}(g)) -
\avol(\pmb{a} \cdot \overline{\pmb{L}} + \pmb{\delta} \cdot \overline{\pmb{A}} +\overline{\OO}(g)) \vert \\\hspace{10em} 
\leq d C_1 \Vert  \pmb{a} \cdot \pmb{\phi} +  \pmb{\delta} \cdot \pmb{\psi}\Vert_{\sup} \vert (\pmb{a}, \pmb{\delta})\vert^{d-1}_1
\end{cases}
\]
Note that
\[
dC_1\Vert  \pmb{a} \cdot \pmb{\phi} \Vert_{\sup} \vert \pmb{a} \vert^{d-1}_1
 \leq  dC_1\vert \pmb{a} \vert^d_{1} \frac{\epsilon}{3C_1d(M+ 1)^d} \leq
 \epsilon/3
\]
and
\[
dC_1\Vert  \pmb{a} \cdot \pmb{\phi} + \pmb{\delta} \cdot \pmb{\psi}\Vert_{\sup} \vert (\pmb{a}, \pmb{\delta})\vert^{d-1}_1
 \leq dC_1 \vert (\pmb{a}, \pmb{\delta}) \vert_{1}^d \frac{\epsilon}{3C_1d(M + 1)^{d}} \leq
  \epsilon/3
\]
for all $\pmb{a} \in B$ and
$\pmb{\delta} \in \QQ^{r'}$ with $\vert \pmb{\delta} \vert_1 \leq 1$.
Thus we get 
\[
\begin{cases}
\vert \avol(\pmb{a} \cdot \overline{\pmb{L}}^{\pmb{\phi}}) -
\avol(\pmb{a} \cdot \overline{\pmb{L}}) \vert \leq \epsilon/3, \\
\vert \avol(\pmb{a} \cdot \overline{\pmb{L}}^{\pmb{\phi}} + \pmb{\delta} \cdot \overline{\pmb{A}}^{\pmb{\psi}}+\overline{\OO}(g) ) -
\avol(\pmb{a} \cdot \overline{\pmb{L}} + \pmb{\delta} \cdot \overline{\pmb{A}} + \overline{\OO}(g)) \vert \leq \epsilon/3
\end{cases}
\]
for all $\pmb{a} \in B$, $g \in C^0(X)$ and $\pmb{\delta} \in \QQ^{r'}$ with $\vert \pmb{\delta} \vert_1 \leq 1$.

On the other hand, by Corollary~\ref{cor:C:infty:cont}, we can find positive real constants
$\delta$ and $\delta'$ depending only on $B$, $\overline{\pmb{L}}^{\pmb{\phi}}$,
$\overline{\pmb{A}}^{\pmb{\psi}}$ and $X$ such that
\[
\vert \avol(\pmb{a} \cdot \overline{\pmb{L}}^{\pmb{\phi}} + \pmb{\delta} \cdot \overline{\pmb{A}}^{\pmb{\psi}} + \overline{\OO}(g)) - \avol(\pmb{a} \cdot \overline{\pmb{L}}^{\pmb{\phi}}) \vert
\leq \epsilon/3 
\]
for $\pmb{a} \in B$, $\vert \pmb{\delta} \vert_1 \leq \delta$ and
$\Vert g \Vert_{\sup} \leq \delta'$.
Therefore,  our theorem follows because
\[
\begin{split}
& \vert \avol(\pmb{a} \cdot \overline{\pmb{L}} + \pmb{\delta} \cdot \overline{\pmb{A}} + \overline{\OO}(g)) - \avol(\pmb{a} \cdot \overline{\pmb{L}}) \vert \\
&\hspace{8em}\leq 
\vert \avol(\pmb{a} \cdot \overline{\pmb{L}}^{\pmb{\phi}} + \pmb{\delta} \cdot \overline{\pmb{A}}^{\pmb{\psi}}+ \overline{\OO}(g)) - \avol(\pmb{a} \cdot \overline{\pmb{L}} + \pmb{\delta} \cdot \overline{\pmb{A}} + \overline{\OO}(g)) \vert \\
&\hspace{14em}+ 
\vert \avol(\pmb{a} \cdot \overline{\pmb{L}}^{\pmb{\phi}} + \pmb{\delta} \cdot \overline{\pmb{A}}^{\pmb{\psi}}+ \overline{\OO}(g)) - \avol(\pmb{a} \cdot \overline{\pmb{L}}^{\pmb{\phi}}) \vert \\
&\hspace{18em}+  
\vert \avol(\pmb{a} \cdot \overline{\pmb{L}}^{\pmb{\phi}}) - \avol(\pmb{a} \cdot \overline{\pmb{L}})\vert.
\end{split}
\]
\end{proof}

\section{Continuous extension of the arithmetic volume function over $\RR$}
Let $V$ be a vector space over $\QQ$ and $f : V \to \RR$ a weakly continuous function
in the sense of Conventions and terminology~\ref{CT:weak:cont}.
We say $f$ has a {\em weakly continuous extension over $\RR$} if
there is a weakly continuous function $\tilde{f} : V \otimes_{\QQ} \RR \to \RR$ with
$\rest{\tilde{f}}{V} = f$.
Note that if there is a weakly continuous extension $\tilde{f}$ of
$f $, then $\tilde{f}$ is uniquely determined.
The following example shows that
a weakly continuous function over $\QQ$ does not necessarily have a weakly continuous extension
over $\RR$.

\begin{Example}
\label{example:cont:QQ:not:cont:RR}
Let $a$ be a positive irrational number, and let $f : \QQ^2 \to \RR$ be a function
given by
\[
f(x, y) = \begin{cases}
\max \{ \vert x \vert, a \vert y \vert \} & \text{if $x + a y > 0$}, \\
0 & \text{if $x + a y \leq 0$}.
\end{cases}
\]
Then it is easy to see the following:
\begin{enumerate}
\renewcommand{\labelenumi}{(\arabic{enumi})}
\item
$f$ is positively homogeneous of degree $1$.

\item
$f$ is continuous on $\QQ^2$.

\item
$f$ is monotonically increasing, that is,
$f(x, y) \leq f(x', y')$ for all $(x, y), (x', y') \in \QQ^2$ with $x \leq x'$ and $y \leq y'$.

\item
$f$ has
no continuous extension over $\RR$.

\item
$f$ is not uniformly continuous on 
$\{ (x, y) \in \QQ^2 \mid \max \{ \vert x \vert, a \vert y \vert \} \leq 1 \}$.
\end{enumerate}
\end{Example}

The following lemma gives a condition to guarantee a weakly continuous extension over $\RR$.

\begin{Lemma}
\label{lem:cont:extension}
Let $f : V \to \RR$ be a weakly continuous function on
a vector space $V$ over $\QQ$.
Then the following are equivalent:
\begin{enumerate}
\renewcommand{\labelenumi}{\rom{(\arabic{enumi})}}
\item
$f$ has a weakly continuous extension over $\RR$.

\item
$f$ is uniformly continuous
on any bounded set $B$ in any finite dimensional vector subspace of $V$.
\end{enumerate}
Moreover, if $f$ is positively homogenous, then the weakly continuous extension of $f$
is also positively homogeneous.
\end{Lemma}

\begin{proof}
``(1) $\Longrightarrow$ (2)'' is obvious by Heine's theorem.

Let us consider ``(2) $\Longrightarrow$ (1)''.
For a vector subspace $W$ of $V$, 
we denote $\rest{f}{W}$ by $f_W$.

(a) We assume that $W$ is finite dimensional.
Let $\{ a_n \}_{n=1}^{\infty}$ be a Cauchy sequence in $W$.
Then there is a bounded set $B$ in $W$ with $a_n \in B$ for all $n$.
Thus, $\{ f(a_n) \}_{n=1}^{\infty}$ is also a Cauchy sequence because
$\rest{f}{B}$ is uniformly continuous.
Hence, by using the well-known way, there is a continuous function 
$\tilde{f}_{W} : W\otimes_{\QQ} \RR \to \RR$ with
$\rest{\tilde{f}_{W}}{W} = f_W$. Namely, if $x \in W\otimes_{\QQ} \RR$ and
$\{ a_n \}_{n=1}^{\infty}$ is a Cauchy sequence in $W$ with
$x = \lim\limits_{n\to\infty} a_n$, then
$\tilde{f}_W(x) = \lim\limits_{n\to\infty} f_W(a_n)$.

(b) Let $W \subseteq W'$ be finite dimensional vector subspaces of $V$.
Then 
\[
\rest{\tilde{f}_{W'}}{W\otimes_{\QQ} \RR} = \tilde{f}_W
\]
because
a Cauchy sequence in $W$ is a Cauchy sequence in $W'$.

Let $x \in V\otimes_{\QQ} \RR$. Then there is a finite dimensional vector space
$W$ of $V$ with $x \in W \otimes_{\QQ} \RR$.
The above (b)  shows that $\tilde{f}_W(x)$ does not depend on the choice of
$W$, so that $\tilde{f}(x)$ is defined by $\tilde{f}_W(x)$.

We need to show that $\tilde{f} : V\otimes_{\QQ} \RR \to \RR$ is weakly continuous.
Let $T$ be a finite dimensional vector subspace of $V\otimes_{\QQ} \RR$.
Then there is a finite dimensional vector subspace $W$ of $V$
with $T \subseteq W\otimes_{\QQ} \RR$.
Note that $\rest{\tilde{f}}{W\otimes_{\QQ} \RR} = \tilde{f}_W$.
Thus $\rest{\tilde{f}}{T}$ is continuous.

The last assertion is obvious by our construction.
\end{proof}

Let $X$ be a $d$-dimensional projective arithmetic variety.
The exact sequence \eqref{eqn:exact:C0:aPic:Pic:QQ} gives rise to
the following exact sequence:
\[
C^0(X) \otimes_{\QQ} \RR \to \aPic_{C^0}(X) \otimes_{\ZZ} \RR \to \Pic(X) \otimes_{\ZZ} \RR \to 0.
\]
Let us consider the natural homomorphisms
$\mu : C^0(X) \otimes_{\QQ} \RR \to C^0(X)$  given by $\mu(f \otimes x) = x f$. 
Let $N$ be the image of $\Ker(\mu)$ via $C^0(X) \otimes_{\QQ} \RR \to \aPic_{C^0}(X) \otimes_{\ZZ} \RR$.
We set 
\[
\aPic_{C^0}(X)_{\RR} = (\aPic_{C^0}(X) \otimes_{\ZZ} \RR)/N
\]
and the canonical homomorphism $\aPic_{C^0}(X) \otimes_{\ZZ} \RR \to \aPic_{C^0}(X)_{\RR}$
is denoted by $\pi$.
Then the above exact sequence yields the following commutative diagram:
\addtocounter{Theorem}{1}
\begin{equation}
\label{eqn:exact:C0:aPic:Pic:RR}
\begin{CD}
C^0(X) @>>> \aPic_{C^0}(X) \otimes_{\ZZ} \QQ @>>> \Pic(X) \otimes_{\ZZ} \QQ @>>> 0 \\
@| @VVV @VVV \\
C^0(X) @>>> \aPic_{C^0}(X)_{\RR} @>>> \Pic(X) \otimes_{\ZZ} \RR @>>> 0,
\end{CD}
\end{equation}
where each horizontal sequence is exact.

The following theorem is one of the main theorem of this paper.

\begin{Theorem}
\label{thm:cont:extension:aPic}
\begin{enumerate}
\renewcommand{\labelenumi}{\rom{(\arabic{enumi})}}
\item
There is a unique weakly continuous and positively homogeneous function $\avol : \aPic_{C^0}(X) \otimes_{\ZZ} \RR \to \RR$
of degree $d$, which is a continuous extension of
$\avol : \aPic_{C^0}(X) \otimes_{\ZZ} \QQ \to \RR$ over $\RR$.

\item
The above arithmetic volume function $\avol : \aPic_{C^0}(X) \otimes_{\ZZ} \RR \to \RR$
descend to $\aPic_{C^0}(X)_{\RR} \to \RR$, that is,
there is a weakly continuous and positively homogeneous function 
\[
\avol' :  \aPic_{C^0}(X)_{\RR} \to \RR
\]
of degree $d$ such that the following diagram is commutative:
\[
\xymatrix{
\aPic_{C^0}(X) \otimes_{\ZZ}\RR \ar[r]^(.68){\avol} \ar[d]_{\pi} & \RR \\
\aPic_{C^0}(X)_{\RR} \ar[ru]_{\avol'} & \\
}
\]
By abuse of notation, $\avol'$ is also denoted by $\avol$.
\end{enumerate}
\end{Theorem}

\begin{proof}
The first assertion follows from Theorem~\ref{thm:cont:C:0} and Lemma~\ref{lem:cont:extension}.
For the second assertion, let us consider the following claim:

\begin{Claim}
Every element of $N$ can be written by a form
\[
\overline{\OO}(f_1) \otimes x_1 + \cdots + \overline{\OO}(f_r) \otimes x_r
\]
for $f_1,\ldots,f_r \in C^0(X)$ and $x_1,\ldots,x_r \in \RR$ with $x_1f_1+\cdots+x_rf_r = 0$.
\end{Claim}

\begin{proof}
For $\omega = f_1 \otimes x_1 + \cdots + f_r \otimes x_r \in C^0(X) \otimes_{\QQ} \RR$,
$\omega \in \Ker(\mu)$ if and only if $x_1 f_1 + \cdots + x_r f_r = 0$,
which proves the claim.
\end{proof}

Let $\overline{L} \in \aPic_{C^0}(X) \otimes \RR$,
$f_1,\ldots,f_r \in C^0(X)$ and $x_1,\ldots,x_r \in \RR$ with
$x_1f_1+\cdots+x_rf_r = 0$.
It is sufficient to show that
\[
\avol\left(\overline{L} + \overline{\OO}(f_1) \otimes x_1 + \cdots + \overline{\OO}(f_r) \otimes x_r \right) = \avol(\overline{L}).
\]
Let us choose a sequence $\{ \pmb{x}_n \}_{n=1}^{\infty}$ in $\QQ^r$ with
$\lim_{n\to\infty} \pmb{x}_n = (x_1,\ldots,x_r)$.
Then, by Theorem~\ref{thm:cont:extension:aPic},
\begin{multline*}
\lim_{n\to\infty}
\avol\left(\overline{L} + \overline{\OO}(f_1) \otimes \pmb{x}_n(1) + \cdots + \overline{\OO}(f_r) \otimes \pmb{x}_n(r)\right) \\
=
\avol\left(\overline{L} + \overline{\OO}(f_1) \otimes x_1 + \cdots + \overline{\OO}(f_r) \otimes x_r\right).
\end{multline*}
Since $\pmb{x}_n \in \QQ^r$, if we set $\phi_n = \pmb{x}_n(1)f_1 + \cdots + \pmb{x}_n(r) f_r$,
then 
\[
\overline{\OO}(f_1) \otimes \pmb{x}_n(1) + \cdots + \overline{\OO}(f_r) \otimes \pmb{x}_n(r)
= \overline{\OO}(\phi_n).
\]
Note that
\begin{align*}
\Vert \phi_n \Vert_{\sup} & =
\Vert (\pmb{x}_n(1) - x_1) f_1 + \cdots + (\pmb{x}_n(r) - x_r) f_r \Vert_{\sup} \\
& \leq
\vert \pmb{x}_n(1) - x_1 \vert \Vert f_1 \Vert_{\sup}
+ \cdots + \vert \pmb{x}_n(r) - x_r\vert \Vert f_r \Vert_{\sup}.
\end{align*}
Thus $\lim_{n\to\infty} \Vert \phi_n \Vert_{\sup} = 0$. On the other hand, by (1) of
Proposition~\ref{cor:thm:cont:extension:aPic},
\[
\vert \avol(\overline{L} + \overline{\OO}(\phi_n)) - \avol(\overline{L}) \vert \leq
d \Vert \phi_n \Vert_{\sup} \vol(L_{\QQ}).
\]
Therefore, 
\begin{multline*}
\lim_{n\to\infty} \avol\left(\overline{L} + \overline{\OO}(f_1) \otimes \pmb{x}_n(1) + \cdots + \overline{\OO}(f_r) \otimes \pmb{x}_n(r)\right) \\
= \lim_{n\to\infty} \avol(\overline{L} + \overline{\OO}(\phi_n)) =
\avol(\overline{L}).
\end{multline*}
Thus we get (2).
\end{proof}

\begin{Example}
Let $K = \QQ(\sqrt{2})$, $\OO_K = \ZZ[\sqrt{2}]$ and $X = \Spec(\OO_K)$.
Note that $\{ \sigma : K \hookrightarrow \CC\} = \{\sigma_1, \sigma_2 \}$,
where $\sigma_1(\sqrt{2}) = \sqrt{2}$ and $\sigma_2(\sqrt{2}) = -\sqrt{2}$.
Then $C^0(X) = \RR^2$ in the natural way.
Moreover, the class number of $K$ is $1$ and the fundamental unit of $K$
is $\sqrt{2} + 1$. Thus, if we set  $\omega = \left(\log(\sqrt{2} + 1), \log(\sqrt{2} - 1)\right)$,
then we have an exact sequence
\[
0 \to \ZZ \omega \to \RR^2 \overset{\overline{\OO}}{\longrightarrow} \aPic_{C^0}(X) \to 0,
\]
which yields the following commutative diagram:
\[
\begin{CD}
0 @>>>  \QQ \omega @>>> \RR^2 @>{\overline{\OO}}>> \aPic_{C^0}(X) \otimes_{\ZZ} \QQ @>>> 0 \\
@. @VVV @|  @VVV @. \\
0 @>>>  \RR \omega @>>> \RR^2 @>{\overline{\OO}}>> \aPic_{C^0}(X)_{\RR} @>>> 0,
\end{CD}
\]
where each horizontal sequence is exact. In particular,
the canonical homomorphism $\aPic_{C^0}(X) \otimes_{\ZZ} \QQ  \to \aPic_{C^0}(X)_{\RR}$ is
not injective.
Moreover, it is easy to see that
\[
\avol(\overline{\OO}(\lambda_1, \lambda_2)) = \begin{cases}
\lambda_1 + \lambda_2 & \text{if $\lambda_1 + \lambda_2 \geq 0$}, \\
0 & \text{otherwise}.
\end{cases}
\]
\end{Example}

The arithmetic volume function
$\avol : \aPic_{C^0}(X)_{\RR} \to \RR$ has the following properties:

\begin{Proposition}
\label{cor:thm:cont:extension:aPic}
\begin{enumerate}
\renewcommand{\labelenumi}{\rom{(\arabic{enumi})}}
\item
For all $\overline{L} \in \aPic_{C^0}(X)_{\RR}$ and $f \in C^0(X)$, we have
\[
\left\vert \avol(\overline{L} + \overline{\OO}(f)) -
\avol(\overline{L} )\right\vert \leq
d \Vert f \Vert_{\sup} \vol(L_{\QQ}).
\]

\item
Let $\nu : X' \to X$ be a morphism of projective arithmetic varieties.
Then there is a unique homomorphism $\nu^* : \aPic_{C^0}(X)_{\RR} \to \aPic_{C^0}(X')_{\RR}$
such that the following diagram is commutative:
\[
\begin{CD}
\aPic_{C^0}(X) \otimes_{\ZZ} \RR @>{\nu^* \otimes \operatorname{id}} >> \aPic_{C^0}(X') \otimes_{\ZZ} \RR \\
@V{\pi}VV @V{\pi}VV \\
\aPic_{C^0}(X)_{\RR} @>{\nu^*}>> \aPic_{C^0}(X')_{\RR}.
\end{CD}
\]
Moreover,  if $\nu$ is birational, then $\avol(\nu^*(\overline{L})) = \avol(\overline{L})$ for
$\overline{L} \in \aPic_{C^0}(X)_{\RR}$.

\item
Let $\overline{L}_1, \ldots, \overline{L}_r,
\overline{A}_1, \ldots, \overline{A}_{r'}$ be 
$C^{\infty}$-hermitian
$\QQ$-invertible sheaves on $X$.
If we set $\overline{\pmb{L}} = 
(\pi(\overline{L}_1), \ldots, \pi(\overline{L}_r))$ and
$\overline{\pmb{A}} =
(\pi(\overline{A}_1), \ldots, \pi(\overline{A}_{r'}))$, then
there is a positive constant $C$
depending only on $X$ and $\overline{L}_1, \ldots, \overline{L}_r,
\overline{A}_1, \ldots, \overline{A}_{r'}$
such that
\[
\vert \avol(\pmb{a} \cdot \overline{\pmb{L}} +
\pmb{\delta} \cdot \overline{\pmb{A}})
- \avol(\pmb{a} \cdot \overline{\pmb{L}}) \vert
\leq C \vert (\pmb{a}, \pmb{\delta}) \vert_{1}^{d-1}\vert \pmb{\delta} \vert_{1}
\]
for all $\pmb{a} \in \RR^{r}$ and
$\pmb{\delta} \in \RR^{r'}$.

\item
Let $V$ be a finite dimensional vector subspace of $\aPic_{C^0}(X)_{\RR}$ and
$\Vert\cdot\Vert$ a norm of $V$.
Let $K$ be a compact set in $V$.
For any positive real number $\epsilon$,
there are positive real number $\delta$ and $\delta'$ such that
\[
\vert \avol(x+ a + \overline{\OO}(g)) - \avol(x) \vert
\leq \epsilon
\]
for all $x \in K$, $a \in V$ and $g \in C^0(X)$
with $\Vert a \Vert \leq
\delta$ and $\Vert g \Vert_{\sup} \leq \delta'$.

\item
Let $\{ x_n \}_{n=1}^{\infty}$ be a sequence in 
a finite dimensional vector subspace of $\aPic_{C^0}(X)_{\RR}$ and
$\{ f_n \}_{n=1}^{\infty}$ a sequence in $C^0(X)$ such that
$\{ x_n \}_{n=1}^{\infty}$ converges to $x$ in the usual topology and
$\{ f_n \}_{n=1}^{\infty}$ converges uniformly to $f$.
Then
\[
\lim_{n\to\infty} \avol\left(x_n
+ \overline{\OO}(f_n)\right)
= \avol\left(x 
+ \overline{\OO}(f)\right).
\]

\item
Let $\overline{L}_1, \ldots, \overline{L}_r$ be $\QQ$-effective continuous hermitian $\QQ$-invertible sheaves
on $X$.
For $(a_1, \ldots, a_r), (a'_1, \ldots, a'_r) \in \RR^r$ and $h, h' \in C^0(X)$,
if $a_i \leq a'_i$ \rom{(}$\forall i$\rom{)} and $h \leq h'$, then
\[
\hspace{4em}
\avol(a_1 \pi(\overline{L}_1) + \cdots a_r \pi(\overline{L}_r) + \overline{\OO}(h))
\leq
\avol(a'_1 \pi(\overline{L}_1) + \cdots a'_r \pi(\overline{L}_r) + \overline{\OO}(h')).
\]
\end{enumerate}
\end{Proposition}

\begin{proof}
(1) It follows from (1) of Proposition~\ref{prop:avol:hom:cont:QQ} and Theorem~\ref{thm:cont:extension:aPic}.

(2) Let $f_1, \ldots, f_r \in C^0(X)$ and $x_1, \ldots, x_r \in \RR$ with
$x_1 f_1 + \cdots + x_r f_r = 0$.
Then
\[
(\nu^* \otimes \operatorname{id})\left(\sum \overline{\OO}(f_i) \otimes x_i \right)
= \sum \overline{\OO}(\nu^*(f_i)) \otimes x_i 
\]
and
\[
x_1 \nu^*(f_1) + \cdots + x_r \nu^*(f_r) = \nu^*(x_1 f_1 + \cdots + x_r f_r) = 0.
\]
This observation shows the first assertion.
The second assertion is a consequence of (3) of Proposition~\ref{prop:avol:hom:cont:QQ} and Theorem~\ref{thm:cont:extension:aPic}.

(3) It is implied by Theorem~\ref{thm:C:infty:cont} and Theorem~\ref{thm:cont:extension:aPic}.

(4) We can find a finite dimensional vector subspace $W$ of $\aPic_{C^0}(X) \otimes_{\ZZ} \QQ$ such that
$V \subseteq \pi(W \otimes_{\QQ} \RR)$.
Thus it follows from
Theorem~\ref{thm:cont:C:0} and Theorem~\ref{thm:cont:extension:aPic}.

(5) Let $V$ be a vector space generated by $\{ x_n \}_{n=1}^{\infty}$ and $\overline{\OO}(f)$.
Note that $V$ is finite dimensional and
$x_n + \overline{\OO}(f_n) = x + \overline{\OO}(f) + (x_n-x) + \overline{\OO}(f_n - f)$.
Thus, applying (4)  to $V$,
we can see (5)

(6) This can be proved by (2) of Proposition~\ref{prop:ineq:Pic:C:0} and Theorem~\ref{thm:cont:extension:aPic}.
\end{proof}

\renewcommand{\theequation}{\arabic{section}.\arabic{Theorem}.\arabic{Claim}}

\section{Approximation of the arithmetic volume function}
Let $X$ be a $d$-dimensional projective arithmetic variety.
The purpose of this section is to prove the following theorem,
which gives an approximation of the arithmetic volume function in term
of $\ah$.

\begin{Theorem}
\label{thm:vol:limit:arith:var}
Let $M$ be a finitely generated $\ZZ$-submodule of $\Pic_{C^0}(X)$.
Let $\{ \overline{A}_n \}_{n=1}^{\infty}$ be a sequence in $M$ and
$\{ f_n \}_{n=1}^{\infty}$ a sequence in $C^0(X)$ such that
$\{ \overline{A}_n \otimes 1/n\}_{n=1}^{\infty}$ converges  to $\overline{A}$ in $M \otimes \RR$
in the usual topology and 
$\{ f_n/n \}_{n=1}^{\infty}$ converges uniformly to $f$.
Then
\[
\lim_{n\to\infty} \frac{\ah(\overline{L}_n + \overline{\OO}(f_n)) }{n^d/d!} \\
=
\avol( \overline{A} + \overline{\OO}(f)).
\]
\end{Theorem}

\begin{proof}
Let $\overline{L}_1, \ldots, \overline{L}_l$ be a generator of $M$
such that $\overline{L}_1, \ldots, \overline{L}_r$ ($r \leq l$) gives rise to
a free basis of $M/M_{tor}$ and
that $\overline{L}_{r+1}, \ldots, \overline{L}_l$ are torsion elements of $M$.
Here we set $\overline{\pmb{L}} = (\overline{L}_1, \ldots, \overline{L}_l)$.
Let $N$ be a positive integer such that
$N \overline{L}_{r+1} = \cdots = N \overline{L}_l = 0$.
Then we can find $\pmb{a}_n \in \ZZ^l$ such that
$\overline{A}_n = \pmb{a}_n \cdot \overline{\pmb{L}}$ and
$0 \leq \pmb{a}_n(i) \leq N$ for all $i = r+1, \ldots, l$.
By our assumption, for $1 \leq i \leq r$,
$\{ \pmb{a}_n(i)/n \}_{n=1}^{\infty}$ converges to $a_i \in \RR$.
Thus, if we set $\pmb{a} =(a_1, \ldots, a_r, 0, \ldots, 0)$,
then $\lim_{n\to\infty} \pmb{a}_n/n = \pmb{a}$.
Therefore, it is sufficient to show the following theorem.
\end{proof}

\begin{Theorem}
\label{thm:vol:limit:arith:var:L}
Let $\{ \pmb{a}_n \}_{n=1}^{\infty}$ be a sequence in 
$\ZZ^l$ and
$\{ f_n \}_{n=1}^{\infty}$ a sequence in $C^0(X)$ such that
\[
\pmb{a} = \lim\limits_{n\to\infty}\pmb{a}_n/n \in \RR^l\quad\text{and}\quad
\lim_{n\to\infty} \Vert (f_n/n) - f \Vert_{\sup} = 0
\]
for some $f \in C^0(X)$.
Then,
for a finite sequence $\overline{\pmb{L}} = (\overline{L}_1, \ldots, \overline{L}_l)$ in  $\aPic_{C^0}(X)$,
\[
\lim_{n\to\infty} \frac{\ah(\pmb{a}_n \cdot \overline{\pmb{L}} + \overline{\OO}(f_n))}{n^d/d!}
= \avol(\pmb{a} \cdot \overline{\pmb{L}} + \overline{\OO}(f)).
\]
\end{Theorem}

\begin{proof}
First of all, let us see the following claim:

\begin{Claim}
\label{claim:thm:vol:limit:arith:var:smooth:3}
We may assume that $\overline{\pmb{L}}$ is effective, that is,
$\overline{L}_i$ is effective for every $i$.
\end{Claim}

\begin{proof}
We can find $\overline{L}'_i \geq 0$ and $\overline{L}''_i \geq 0$ with
$\overline{L}_i = \overline{L}'_i - \overline{L}''_i$.
We set $\pmb{a}' $, $\pmb{a}'_n$ and $\overline{\pmb{L}}'$ as follows:
\[
\begin{cases}
\pmb{a}' = (\pmb{a}(1), \ldots, \pmb{a}(l), -\pmb{a}(1), \ldots, -\pmb{a}(l)), \\
\pmb{a}'_n = (\pmb{a}_n(1), \ldots, \pmb{a}_n(l), -\pmb{a}_n(1), \ldots, -\pmb{a}_n(l)), \\
\overline{\pmb{L}}' = (\overline{L}'_1, \ldots, \overline{L}'_l, \overline{L}''_1, \ldots, \overline{L}''_l).
\end{cases}
\]
Then $\pmb{a} \cdot \overline{\pmb{L}} = \pmb{a}' \cdot \overline{\pmb{L}}'$,
$\pmb{a}_n \cdot \overline{\pmb{L}}  = \pmb{a}'_n \cdot \overline{\pmb{L}}'$ and
$\lim_{n\to\infty} \pmb{a}'_n/n = \pmb{a}'$.
Thus the claim follows.
\end{proof}

Under the assumption  that $\overline{\pmb{L}}$ is effective,
we will prove this theorem in the following steps:

\begin{enumerate}[{Step} 1.]
\item
If $X$ is generically smooth, $f$ is $C^{\infty}$ and
$\overline{\pmb{L}} = (\overline{L}_1, \ldots, \overline{L}_l)$ is a finite sequence of $C^{\infty}$-hermitian
invertible sheaves on $X$, then the assertion of Theorem~\rom{\ref{thm:vol:limit:arith:var:L}} holds.

\item
If $X$ is generically smooth, then the assertion of Theorem~\rom{\ref{thm:vol:limit:arith:var:L}} holds.

\item
If $X$ is normal, then the assertion of Theorem~\rom{\ref{thm:vol:limit:arith:var:L}} holds.

\item
In general,
Theorem~\rom{\ref{thm:vol:limit:arith:var:L}} holds.
\end{enumerate}

\bigskip
{\bf Step 1}:
Let us begin with the following claim:

\begin{Claim}
\label{claim:thm:vol:limit:arith:var:smooth:2}
Let $\overline{L}$ and $\overline{A}$ be $C^{\infty}$-hermitian invertible sheaves on $X$.
Then there are  positive constants $C$, $D$ and $n_1$ depending only on $\overline{L}$,
$\overline{A}$ and $X$ such that
\[
\begin{cases}
 \ah\left(n\overline{L}+ \lceil n\epsilon \rceil \overline{A}\right)
 \leq  \ah\left(n\overline{L}\right) + C\lceil n\epsilon\rceil n^{d-1} + D n^{d-1} \log(n) \\
\ah\left(n\overline{L} \right) \leq
 \ah\left(n\overline{L}- \lceil n\epsilon \rceil \overline{A}\right)
 + C\lceil n\epsilon\rceil n^{d-1} + D n^{d-1} \log(n)
\end{cases}
\]
for all $n \in \ZZ$ and $\epsilon \in \RR$ with $n \geq n_1$ and $0 \leq \epsilon \leq 1/2$.
\end{Claim}

\begin{proof}
Note that if  $n \geq 2$ and $0 \leq \epsilon \leq 1/2$, 
then $n \geq (n/2) + 1 \geq \lceil n/2 \rceil\geq \lceil n\epsilon\rceil$.
Thus the claim follows from Theorem~\ref{thm:h:0:estimate:big:main} or
\cite[Theorem~3.1]{MoCont}.
\end{proof}

First we consider the case where $\pmb{a} \in \ZZ^l$.
In this case, by \cite{Chen},
\[
\lim_{n\to\infty} \frac{\ah(n (\pmb{a} \cdot \overline{\pmb{L}} + \overline{\OO}(f)))}{n^d/d!}
=\avol(\pmb{a} \cdot \overline{\pmb{L}} + \overline{\OO}(f)).
\]
For any $0 < \epsilon <1/2$,
there is a positive integer $n_0$ such that 
\[
\vert \pmb{a}_n - n\pmb{a} \vert_{1} \leq n \epsilon\quad\text{and}\quad
\Vert f_n - nf \Vert_{\sup} \leq n \epsilon
\]
for all
$n \geq n_0$. Thus if we set  $\pmb{1} = (1, \ldots, 1)$, then
\[
n \pmb{a} - \lceil n \epsilon \rceil \pmb{1} \leq
\pmb{a}_n \leq n \pmb{a} + \lceil n \epsilon \rceil \pmb{1}
\quad\text{and}\quad
n f - \lceil n \epsilon \rceil \leq f_n \leq n f + \lceil n \epsilon \rceil,
\]
where the first inequality means that
$n \pmb{a}(i) - \lceil n \epsilon \rceil \leq
\pmb{a}_n(i) \leq n \pmb{a}(i) + \lceil n \epsilon \rceil$ for all $i$.
Therefore,
\begin{multline*}
\ah\left(n (\pmb{a} \cdot \overline{\pmb{L}} + \overline{\OO}(f)) - \lceil n \epsilon \rceil ( \pmb{1} \cdot \overline{\pmb{L}} + \overline{\OO}(1))\right) 
\leq
\ah\left(\pmb{a}_n \cdot \pmb{L} + \overline{\OO}(f_n)\right) \\
\leq 
\ah\left(n (\pmb{a} \cdot \overline{\pmb{L}} + \overline{\OO}(f)) + \lceil n \epsilon \rceil ( \pmb{1} \cdot \overline{\pmb{L}} + \overline{\OO}(1))\right).
\end{multline*}
Thus, by Claim~\ref{claim:thm:vol:limit:arith:var:smooth:2}, 
there are constant $C$ and $D$ depending only on
$\pmb{a} \cdot \overline{\pmb{L}} + \overline{\OO}(f)$ and $\pmb{1} \cdot \overline{\pmb{L}} + \overline{\OO}(1)$
such that
\begin{multline*}
\ah(n (\pmb{a} \cdot \overline{\pmb{L}} + \overline{\OO}(f))) 
- C\lceil n \epsilon \rceil  n^{d-1} - D n^{d-1}\log(n) 
\leq 
\ah (\pmb{a}_n \cdot \overline{\pmb{L}} + \overline{\OO}(f_n))
\leq \\
\ah(n (\pmb{a} \cdot \overline{\pmb{L}} + \overline{\OO}(f)))
+ C\lceil n \epsilon \rceil  n^{d-1} + D n^{d-1}\log(n).
\end{multline*}
for all $n \gg 1$.
Thus, taking $n \to \infty$, we obtain the following:
\begin{multline*}
\avol(\pmb{a} \cdot \overline{\pmb{L}} + \overline{\OO}(f))  - C d! \epsilon  \leq
 \liminf_{n\to\infty}\frac{\ah(\pmb{a}_n \cdot \overline{\pmb{L}} + \overline{\OO}(f_n))}{n^d/d!} \\
 \leq
 \limsup_{n\to\infty}\frac{\ah(\pmb{a}_n \cdot \overline{\pmb{L}} + \overline{\OO}(f_n))}{n^d/d!} \leq
 \avol(\pmb{a} \cdot \overline{\pmb{L}} + \overline{\OO}(f))  + C d! \epsilon.
\end{multline*}
Here $\epsilon$ is arbitrary. Thus
\[
 \liminf_{n\to\infty}\frac{\ah(\pmb{a}_n \cdot \overline{\pmb{L}} + \overline{\OO}(f_n))}{n^d/d!} = \limsup_{n\to\infty}\frac{\ah(\pmb{a}_n \cdot \overline{\pmb{L}} + \overline{\OO}(f_n))}{n^d/d!} = \avol(\pmb{a} \cdot \overline{\pmb{L}})
 \]
which shows the case where $\pmb{a} \in \ZZ^l$.

Next we consider the case where $\pmb{a} \in \QQ^l$.
Let $N$ be a positive integer with $N \cdot \pmb{a} \in \ZZ^l$.
Since $\lim\limits_{n\to\infty} \pmb{a}_{Nn+k}/n = N \pmb{a}$ and
$\lim_{n\to\infty}\Vert f_{Nn+k}/n - Nf \Vert_{\sup} = 0$ for $0 \leq k < N$,
by using the previous case, we have
\[
\lim_{n\to\infty} \frac{\ah(\pmb{a}_{Nn+k}\cdot \overline{\pmb{L}} + \overline{\OO}(f_{Nn+k}))}{n^d/d!} =
\avol(N\pmb{a}\cdot\overline{\pmb{L}} + \overline{\OO}(Nf)) = N^d \avol(\pmb{a}\cdot\overline{\pmb{L}}+\overline{\OO}(f)).
\]
On the other hand,
\begin{align*}
\lim_{n\to\infty} \frac{\ah(\pmb{a}_{Nn+k}\cdot \overline{\pmb{L}} + \overline{\OO}(f_{Nn+k}))}{n^d/d!} & =
\lim_{n\to\infty} \frac{(Nn+k)^d}{n^d} \frac{\ah(\pmb{a}_{Nn+k}\cdot \overline{\pmb{L}} + \overline{\OO}(f_{Nn+k}))}{(Nn+k)^d/d!} \\
& = N^d \lim_{n\to\infty} \frac{\ah(\pmb{a}_{Nn+k}\cdot \overline{\pmb{L}} + \overline{\OO}(f_{Nn+k}))}{(Nn+k)^d/d!}.
\end{align*}
Thus we get
\[
\lim_{n\to\infty} \frac{\ah(\pmb{a}_{Nn+k}\cdot \overline{\pmb{L}} + \overline{\OO}(f_{Nn+k}))}{(Nn+k)^d/d!} =
\avol(\pmb{a}\cdot\overline{\pmb{L}} + \overline{\OO}(f))
\]
for all $k$ with $0 \leq k < N$,
which proves the case where $\pmb{a} \in \QQ^l$.

Finally we consider a general case.
For $\epsilon > 0$,
let us choose $\pmb{\delta} = (\delta_1,\ldots,\delta_l),
\pmb{\delta}' =  (\delta'_1,\ldots,\delta'_l) \in \RR^l_{\geq 0}$
such that $\pmb{a} + \pmb{\delta},
\pmb{a} - \pmb{\delta}' \in \QQ^l$ and $\vert \pmb{\delta} \vert_{1},
\vert \pmb{\delta}' \vert_{1} \leq \epsilon$.
If we set 
\[
\pmb{b}_n = \pmb{a}_n + ([n\delta_1], \ldots, [n\delta_l])\quad
\text{and}\quad
\pmb{b}'_n = \pmb{a}_n - ([n\delta'_1], \ldots, [n\delta'_l]),
\]
then $\lim\limits_{n\to\infty}\pmb{b}_n/n = \pmb{a} + \pmb{\delta}$
and $\lim\limits_{n\to\infty}\pmb{b}'_n/n = \pmb{a} - \pmb{\delta}'$.
Thus, using the previous case, we have
\begin{multline*}
\avol((\pmb{a} - \pmb{\delta}') \cdot \overline{\pmb{L}} + \overline{\OO}(f)) = \liminf_{n\to\infty}\frac{\ah(\pmb{b}'_n \cdot \overline{\pmb{L}} + \overline{\OO}(f_n))}{n^d/d!} \\
\leq
\liminf_{n\to\infty}\frac{\ah(\pmb{a}_n \cdot \overline{\pmb{L}}+ \overline{\OO}(f_n))}{n^d/d!}
\leq
\limsup_{n\to\infty}\frac{\ah(\pmb{a}_n \cdot \overline{\pmb{L}} + \overline{\OO}(f_n))}{n^d/d!} \\
\leq
\limsup_{n\to\infty}\frac{\ah(\pmb{b}_n \cdot \overline{\pmb{L}}+ \overline{\OO}(f_n))}{n^d/d!} = 
\avol((\pmb{a} + \pmb{\delta}) \cdot \overline{\pmb{L}} + \overline{\OO}(f)).
\end{multline*}
By (6) of Proposition~\ref{cor:thm:cont:extension:aPic},
\[
\avol((\pmb{a} - \epsilon \pmb{1}) \cdot \overline{\pmb{L}}+ \overline{\OO}(f)) \leq \avol((\pmb{a} - \pmb{\delta}') \cdot \overline{\pmb{L}}+ \overline{\OO}(f))
\]
and
\[
\avol((\pmb{a} + \pmb{\delta}) \cdot \overline{\pmb{L}}+ \overline{\OO}(f)) \leq
\avol((\pmb{a} + \epsilon \pmb{1}) \cdot \overline{\pmb{L}}+ \overline{\OO}(f)).
\]
Therefore, 
\begin{multline*}
\avol((\pmb{a} - \epsilon \pmb{1}) \cdot \overline{\pmb{L}}+ \overline{\OO}(f)) \leq 
\liminf_{n\to\infty}\frac{\ah(\pmb{a}_n \cdot \overline{\pmb{L}}+ \overline{\OO}(f_n))}{n^d/d!} \\
\leq
\limsup_{n\to\infty}\frac{\ah(\pmb{a}_n \cdot \overline{\pmb{L}}+ \overline{\OO}(f_n))}{n^d/d!}
\leq
\avol((\pmb{a} + \epsilon \pmb{1}) \cdot \overline{\pmb{L}}+ \overline{\OO}(f)).
\end{multline*}
Thus,
taking $\epsilon \to 0$ and
using the continuity of the volume function,
we have
\[
\avol(\pmb{a} \cdot \overline{\pmb{L}}+ \overline{\OO}(f)) = \liminf_{n\to\infty}\frac{\ah(\pmb{a}_n \cdot \overline{\pmb{L}}+ \overline{\OO}(f_n))}{n^d/d!}
=\limsup_{n\to\infty}\frac{\ah(\pmb{a}_n \cdot \overline{\pmb{L}}+ \overline{\OO}(f_n))}{n^d/d!}.
\]
Hence we get Step~1.

\bigskip
{\bf Step 2}:
It is sufficient to show the following inequality:
\addtocounter{Claim}{1}
\begin{multline}
\label{eqn:thm:vol:limit:arith:var:step:2:1}
\avol(\pmb{a} \cdot \overline{\pmb{L}} + \overline{\OO}(f)) - 2d \epsilon (\vert \pmb{a} \vert_{1}   + 1) \vol((\pmb{a} + \pmb{1})\cdot \pmb{L}_{\QQ}) \\
\leq
\liminf_{n\to\infty} \frac{\ah(\pmb{a}_n \cdot \overline{\pmb{L}} + \overline{\OO}(f_n))}{n^d/d!}
\leq
\limsup_{n\to\infty} \frac{\ah(\pmb{a}_n \cdot \overline{\pmb{L}} + \overline{\OO}(f_n))}{n^d/d!} \\
\leq
\avol(\pmb{a} \cdot \overline{\pmb{L}} + \overline{\OO}(f)) + 2d \epsilon (\vert \pmb{a} \vert_{1}   + 1) \vol((\pmb{a} + \pmb{1})\cdot \pmb{L}_{\QQ})
\end{multline}
for any positive real number $\epsilon$.
By Lemma~\ref{lem:Stone:Weierstrass},
there are $g_1, \ldots, g_l, h \in C^0(X)$ such that
$\Vert g_i \Vert_{\sup} \leq \epsilon$ ($i=1, \ldots, l$), $\Vert h \Vert_{\sup} \leq \epsilon$, 
$f + h$ is $C^{\infty}$ and that
\[
\overline{\pmb{L}}^{\pmb{g}}= (\overline{L}_1 + \overline{\OO}(g_1), \ldots, 
\overline{L}_l+ \overline{\OO}(g_l))
\]
is $C^{\infty}$. Then it is easy to see that
\begin{multline*}
\pmb{a}_n \cdot \overline{\pmb{L}} + \overline{\OO}(f_n) +
\overline{\OO}(-\epsilon(\vert \pmb{a}_n \vert_{1}  + n))
\leq
\pmb{a}_n \cdot \overline{\pmb{L}}^{\pmb{g}} + \overline{\OO}(f_n + nh) \\
\leq
\pmb{a}_n \cdot \overline{\pmb{L}} +  \overline{\OO}(f_n) +
\overline{\OO}(\epsilon(\vert \pmb{a}_n \vert_{1}  + n)),
\end{multline*}
which implies that
\begin{multline*}
\ah\left(
\pmb{a}_n \cdot \overline{\pmb{L}} + \overline{\OO}(f_n) +
\overline{\OO}(-\epsilon(\vert \pmb{a}_n \vert_{1}  + n))
\right)
\leq
\ah\left(\pmb{a}_n \cdot \overline{\pmb{L}}^{\pmb{g}} + \overline{\OO}(f_n + nh)\right) \\
\leq
\ah\left(\pmb{a}_n \cdot \overline{\pmb{L}} +  \overline{\OO}(f_n) +
\overline{\OO}( \epsilon(\vert \pmb{a}_n \vert_{1}  + n))\right).
\end{multline*}
For each $i$, we choose an integer $b_i$ with $\pmb{a}(i) < b_i \leq \pmb{a}(i)+1$.
Then there is a positive integer $n_0$ such that
$\pmb{a}_n(i) \leq n b_i$ for all $n \geq n_0$ and $i$. Thus, if we set
$\pmb{b} = (b_1, \ldots, b_l)$, then
$\pmb{a}_n \leq n \pmb{b}$ for all $n \geq n_0$ and
$\pmb{b} \leq \pmb{a} + \pmb{1}$ .
Thus $h^0(\pmb{a}_n \cdot \pmb{L}_{\QQ}) \leq h^0(n\pmb{b} \cdot \pmb{L}_{\QQ})$
for $n \geq n_0$. 
Hence, by using \cite[(3) of Proposition~2.1]{MoCont},
if we set 
\[
\beta(n) =  \epsilon(\vert \pmb{a}_n \vert_{1}  + n)
h^0(n \pmb{b} \cdot \pmb{L}_{\QQ}) + C_1 n^{d-1}\log(n)
\]
for some positive constant $C_1$, then
\[
\begin{cases}
\ah\left(\pmb{a}_n \cdot \overline{\pmb{L}} +  \overline{\OO}(f_n) +
\overline{\OO}( \epsilon(\vert \pmb{a}_n \vert_{1}  + n))\right) \leq
\ah\left(\pmb{a}_n \cdot \overline{\pmb{L}} +  \overline{\OO}(f_n) \right) + \beta(n),\\
\ah\left(\pmb{a}_n \cdot \overline{\pmb{L}} +  \overline{\OO}(f_n) +
\overline{\OO}( -\epsilon(\vert \pmb{a}_n \vert_{1}  + n))\right) \geq
\ah\left(\pmb{a}_n \cdot \overline{\pmb{L}} +  \overline{\OO}(f_n) \right) - \beta(n)
\end{cases}
\]
for $n \gg 1$. Thus,
\[
- \beta(n)
\leq
\ah\left(\pmb{a}_n \cdot \overline{\pmb{L}} + \overline{\OO}(f_n)\right)
- \ah\left(\pmb{a}_n \cdot \overline{\pmb{L}}^{\pmb{g}} + \overline{\OO}(f_n + nh)\right) 
\leq \beta(n)
\]
for $n \gg 1$.
Therefore, since
\[
\begin{cases}
{\displaystyle \lim_{n\to\infty} \frac{ \ah\left(\pmb{a}_n \cdot \overline{\pmb{L}}^{\pmb{g}} + \overline{\OO}(f_n + nh)\right)}{n^d/d!} = \avol(\pmb{a} \cdot \overline{\pmb{L}}^{\pmb{g}} + \overline{\OO}(f+h))} \quad
(\because \text{Step~1}), \\
{\displaystyle \lim_{n\to\infty} \frac{ \beta(n) }{n^{d}/d!} =
d \epsilon(\vert \pmb{a} \vert_{1}  + 1)
\vol(\pmb{b} \cdot \pmb{L}_{\QQ})}, \\
\vol(\pmb{b} \cdot \pmb{L}_{\QQ}) \leq \vol((\pmb{a} + \pmb{1}) \cdot \pmb{L}_{\QQ}),
\end{cases}
\]
we have
\begin{multline*}
\avol(\pmb{a} \cdot \overline{\pmb{L}}^{\pmb{g}} + \overline{\OO}(f+h)) - d \epsilon (\vert \pmb{a} \vert_{1}   + 1) \vol((\pmb{a} + \pmb{1})\cdot \pmb{L}_{\QQ}) \\
\leq
\liminf_{n\to\infty} \frac{\ah(\pmb{a}_n \cdot \overline{\pmb{L}} + \overline{\OO}(f_n))}{n^d/d!}
\leq
\limsup_{n\to\infty} \frac{\ah(\pmb{a}_n \cdot \overline{\pmb{L}} + \overline{\OO}(f_n))}{n^d/d!} \\
\leq
\avol(\pmb{a} \cdot \overline{\pmb{L}}^{\pmb{g}} + \overline{\OO}(f+h)) + d \epsilon (\vert \pmb{a} \vert_{1}   + 1) \vol((\pmb{a} + \pmb{1})\cdot \pmb{L}_{\QQ}).
\end{multline*}
On the other hand,
by (1) of Proposition~\ref{cor:thm:cont:extension:aPic},
\begin{multline*}
\left\vert \avol(\pmb{a} \cdot \overline{\pmb{L}}^{\pmb{g}} + \overline{\OO}(f + h)) - \avol(\pmb{a} \cdot \overline{\pmb{L}} + \overline{\OO}(f)) \right\vert \\
=
\left\vert \avol((\pmb{a},1) \cdot (\overline{\pmb{L}},\overline{\OO}(f))  + \overline{\OO}(\pmb{a} \cdot \pmb{g} + h)) - \avol((\pmb{a},1) \cdot (\overline{\pmb{L}}, \overline{\OO}(f))) \right\vert \\
\leq d \epsilon (\vert \pmb{a} \vert_1  + 1) \vol((\pmb{a}, 1)\cdot (\pmb{L}_{\QQ}, 0))
\leq  d \epsilon (\vert \pmb{a} \vert_1  + 1) \vol((\pmb{a} + \pmb{1})\cdot \pmb{L}_{\QQ}).
\end{multline*}
Hence \eqref{eqn:thm:vol:limit:arith:var:step:2:1} follows.

\bigskip
{\bf Step 3}:
Let $\nu : X' \to X$ be a generic resolution of singularities of $X$ such that
$X'$ is normal.
Then, since $\nu_* \OO_{X'} = \OO_X$, we have
\[
H^0(X, \pmb{a}_n \cdot \pmb{L}) = H^0(X', \pmb{a}_n \cdot \nu^*(\pmb{L})).
\]
Thus
$\ah(\pmb{a}_n \cdot \overline{\pmb{L}} + \overline{\OO}(f_n)) = \ah(\pmb{a}_n \cdot \nu^*(\overline{\pmb{L}})+\overline{\OO}(\nu^*(f_n)))$.
Therefore, by using Step~2 and (2) of
Proposition~\ref{cor:thm:cont:extension:aPic},
\begin{align*}
\lim_{n\to\infty} \frac{\ah(\pmb{a}_n \cdot \overline{\pmb{L}}+\overline{\OO}(f_n))}{n^d/d!} & =
\lim_{n\to\infty} \frac{\ah(\pmb{a}_n \cdot \nu^*(\overline{\pmb{L}})+\overline{\OO}(\nu^*(f_n)))}{n^d/d!} \\
& =
\avol(\pmb{a} \cdot \nu^*(\overline{\pmb{L}}) + \overline{\OO}(\nu^*(f))) = \avol(\pmb{a} \cdot \overline{\pmb{L}}+\overline{\OO}(f)).
\end{align*}

\bigskip
{\bf Step 4}:
Let $\nu : X' \to X$ be the normalization of $X$.
It is sufficient to see that
\[
\lim_{n\to\infty} \frac{\ah(\pmb{a}_n \cdot \overline{\pmb{L}}+\overline{\OO}(f_n))}{n^d/d!} =
\lim_{n\to\infty} \frac{\ah(\pmb{a}_n \cdot \nu^*(\overline{\pmb{L}})+\overline{\OO}(\nu^*(f_n)))}{n^d/d!}
\]
because, by using (2) of Proposition~\ref{cor:thm:cont:extension:aPic} and Step 3,
the above equation implies that
\begin{align*}
\lim_{n\to\infty} \frac{\ah(\pmb{a}_n \cdot \overline{\pmb{L}}+\overline{\OO}(f_n))}{n^d/d!} & =
\lim_{n\to\infty} \frac{\ah(\pmb{a}_n \cdot \nu^*(\overline{\pmb{L}})+\overline{\OO}(\nu^*(f_n)))}{n^d/d!} \\
& = \avol(\pmb{a} \cdot \nu^*(\overline{\pmb{L}}) + \overline{\OO}(\nu^*(f))) = \avol(\pmb{a} \cdot \overline{\pmb{L}}+\overline{\OO}(f)).
\end{align*}

Since $H^0(X, \pmb{a}_n \cdot \pmb{L}) \subseteq
H^0(X', \pmb{a}_n \cdot \nu^*(\pmb{L}))$,
we have 
\[
\ah(\pmb{a}_n \cdot \overline{\pmb{L}}+\overline{\OO}(f_n)) \leq
\ah(\pmb{a}_n \cdot \nu^*(\overline{\pmb{L}})+\overline{\OO}(\nu^*(f_n))).
\]
Thus
\begin{align*}
\liminf_{n\to\infty} \frac{\ah(\pmb{a}_n \cdot \overline{\pmb{L}}+\overline{\OO}(f_n))}{n^d/d!} & \leq
\limsup_{n\to\infty} \frac{\ah(\pmb{a}_n \cdot \overline{\pmb{L}}+\overline{\OO}(f_n))}{n^d/d!} \\
& \leq
\lim_{n\to\infty} \frac{\ah(\pmb{a}_n \cdot \nu^*(\overline{\pmb{L}})+\overline{\OO}(\nu^*(f_n)))}{n^d/d!}.
\end{align*}
Therefore, we need to show
\[
\lim_{n\to\infty} \frac{\ah(\pmb{a}_n \cdot \nu^*(\overline{\pmb{L}})+\overline{\OO}(\nu^*(f_n)))}{n^d/d!} \leq
\liminf_{n\to\infty} \frac{\ah(\pmb{a}_n \cdot \overline{\pmb{L}}+\overline{\OO}(f_n))}{n^d/d!}.
\]
The proof of the above inequality is similar to one of \cite[Theorem~4.3]{MoCont}.
Let $\mathcal{I}_{X'/X}$ be the conductor ideal sheaf of $X' \to X$.
Let $H$ be an ample invertible sheaf on $X'$ with
$H^0(X', H \otimes \mathcal{I}_{X'/X}) \not= 0$.
Let $s$ be a non-zero element of $H^0(X', H \otimes \mathcal{I}_{X'/X})$.
Let us choose a $C^{\infty}$-hermitian norm $\vert\cdot\vert$ of $H$
with $\Vert s \Vert_{\sup} \leq 1$.
We set $\overline{H} = (H, \vert\cdot\vert)$.

\begin{Claim}
${\displaystyle \lim_{n\to\infty} \frac{\ah(\pmb{a}_n \cdot \nu^*(\overline{\pmb{L}})+\overline{\OO}(\nu^*(f_n)))}{n^d/d!} =
\lim_{n\to\infty} \frac{\ah(\pmb{a}_n \cdot \nu^*(\overline{\pmb{L}}) - \overline{H} + \overline{\OO}(\nu^*(f_n)))}{n^d/d!}}$.
\end{Claim}

\begin{proof}
We set 
\[
\begin{cases}
\overline{\pmb{L}}' = (\nu^*(\overline{L}_1), \ldots, \nu^*(\overline{L}_l), \overline{H}), \\
\pmb{a}'_n = (\pmb{a}_n(1), \ldots, \pmb{a}_n(l), -1), \\
\pmb{a}' = (\pmb{a}(1), \ldots, \pmb{a}(l), 0).
\end{cases}
\]
Then $\pmb{a}_n \cdot \nu^*(\overline{\pmb{L}}) - \overline{H} = \pmb{a}'_n \cdot
\overline{\pmb{L}}'$ and
$\pmb{a}' = \lim\limits_{n\to\infty} \pmb{a}'_n/n$.
By Step~3,
\begin{align*}
\lim_{n\to\infty} \frac{\ah(\pmb{a}_n \cdot \nu^*(\overline{\pmb{L}}) - \overline{H}+\overline{\OO}(\nu^*(f_n)))}{n^d/d!} &=
\lim_{n\to\infty} \frac{\ah(\pmb{a}'_n \cdot \overline{\pmb{L}}' +\overline{\OO}(\nu^*(f_n)))}{n^d/d!} \\
& =
\avol(\pmb{a}' \cdot \overline{\pmb{L}}'+\overline{\OO}(\nu^*(f))) \\
& = 
\avol(\pmb{a} \cdot \nu^*(\overline{\pmb{L}}) + \overline{\OO}(\nu^*(f))) \\
&=
\lim_{n\to\infty} \frac{\ah(\pmb{a}_n \cdot \nu^*(\overline{\pmb{L}})+\overline{\OO}(\nu^*(f_n)))}{n^d/d!}.
\end{align*}
\end{proof}
In the same way as in the proof of \cite[Theorem~4.3]{MoCont}, we can see
\[
\Image\left( H^0(X', \pmb{a}_n \cdot \nu^*(\pmb{L}) - H) \overset{s}{\longrightarrow} 
H^0(X', \pmb{a}_n \cdot \nu^*(\pmb{L}))\right)
\subseteq
H^0(X, \pmb{a}_n \cdot \pmb{L}).
\]
Thus
$\ah(\pmb{a}_n \cdot \nu^*(\overline{\pmb{L}}) - \overline{H} +\overline{\OO}(\nu^*(f_n)))
\leq \ah(\pmb{a}_n \cdot \overline{\pmb{L}}+\overline{\OO}(f_n))$.
Therefore, using the above claim,
\begin{align*}
\lim_{n\to\infty} \frac{\ah(\pmb{a}_n \cdot \nu^*(\overline{\pmb{L}})+\overline{\OO}(\nu^*(f_n)))}{n^d/d!} & =
\lim_{n\to\infty} \frac{\ah(\pmb{a}_n \cdot \nu^*(\overline{\pmb{L}}) - \overline{H}+\overline{\OO}(\nu^*(f_n)))}{n^d/d!} \\
& \leq 
\liminf_{n\to\infty} \frac{\ah(\pmb{a}_n \cdot \overline{\pmb{L}}+\overline{\OO}(f_n))}{n^d/d!}.
\end{align*}

\end{proof}

\end{document}